\documentclass[
  a4paper, 
  reqno, 
  oneside, 
  12pt
]{amsart}

\usepackage[utf8]{inputenc}

\usepackage[dvipsnames]{xcolor} 
\colorlet{cite}{red}
\usepackage{tikz}  
\usetikzlibrary{arrows,positioning}
\tikzset{ 
  baseline=-2.3pt,
  text height=1.5ex, text depth=0.25ex,
  >=stealth,
  node distance=2cm,
  mid/.style={fill=white,inner sep=2.5pt},
}

\usepackage{lmodern}
\usepackage{tikz-cd}
\usepackage{amsthm, amssymb, amsfonts}

\usepackage{amsthm, amssymb, amsfonts}	
\usepackage[all]{xy}
\usepackage{graphicx,caption,subcaption}
\usepackage{braket}  
\usepackage{microtype}
\usepackage[makeroom]{cancel}
\usepackage[%
  bookmarks=true,			
  unicode=true,			
  pdftitle={Methods for construnting special Kahler Lie groups},		%
  pdfauthor={Valencia},	%
  pdfkeywords={Left invariant Kahler structures},	
  colorlinks=true,		
  linkcolor=Red,			
  citecolor=cite,		
  filecolor=magenta,		
  urlcolor=RoyalBlue			
]{hyperref}				
\usepackage{cleveref}

\newtheoremstyle{mydef}
  {}		
  {}		
  {}		
  {}		
  {\scshape}	
  {. }		
  { }		
  {\thmname{#1}\thmnumber{ #2}\thmnote{ #3}}	

\newtheorem{theorem}{Theorem}[section]
\newtheorem*{theorem*}{Theorem}
\newtheorem{proposition}[theorem]{Proposition}
\newtheorem*{proposition*}{Proposition}
\newtheorem{lemma}[theorem]{Lemma}
\newtheorem*{lemma*}{Lemma}
\newtheorem{corollary}[theorem]{Corollary}
\newtheorem*{corollary*}{Corollary}
\theoremstyle{definition}
\newtheorem{definition}[theorem]{Definition}
\newtheorem{Assumption}[theorem]{Assumption}
\newtheorem{example}[theorem]{Example}

\theoremstyle{remark}
\newtheorem{remark}[theorem]{Remark}

\usepackage{tikz}

\author{Fabricio Valencia}

\address{F. Valencia - Departamento de Matem\'atica, Universidade de S\~ao Paulo, Rua do Mat\~ao 1010, S\~ao Paulo, SP 05508-090, Brasil.}
\email{fabricioyarro@gmail.com}

\date{\today}
\title{Left invariant special K\"ahler structures}

\begin{document}
\maketitle

\begin{abstract}
We construct left invariant special K\"ahler structures on the cotangent bundle of a flat pseudo-Riemannian Lie group. We introduce the twisted cartesian product of two special K\"ahler Lie algebras according to two linear representations by infinitesimal K\"ahler transformations. We also exhibit a double extension process of a special K\"ahler Lie algebra which allows us to get all simply connected special K\"ahler Lie groups with bi-invariant symplectic connections. All Lie groups constructed by performing this double extension process can be identified with a subgroup of symplectic (or K\"ahler) affine transformations of its Lie algebra containing a nontrivial  $1$-parameter subgroup formed by central translations. We show a characterization of left invariant flat special K\"ahler structures using \'etale K\"ahler affine representations, exhibit some immediate consequences of the constructions mentioned above, and give several non-trivial examples.
\end{abstract}

\tableofcontents
\section{Introduction}
Throughout this paper we will be dealing with the following geometric object:
\begin{definition}\cite{F}\label{MainDefinition}
	A \emph{special K\"ahler structure} on a smooth manifold $M$ is a triple $(\omega,J,\nabla)$ where $\omega$ is a symplectic form, $J$ is an integrable almost complex structure, and $\nabla$ is a flat and torsion free connection on $M$ such that:	
	\begin{enumerate}
		\item[$\iota.$] $(M,\omega,J)$ is a pseudo-K\"ahler manifold, that is, $k(X,Y)=\omega(X,JY)$ defines a pseudo-Riemannian metric on $M$,
		\item[$\iota\iota.$] $\nabla$ is symplectic with respect to $\omega$, that is, $\nabla \omega=0$ and
		\item[$\iota\iota\iota.$] the following formula holds true
		\begin{equation}\label{Eq1}
		(\nabla_XJ)Y=(\nabla_YJ)X,\qquad X,Y\in\mathfrak{X}(M).
		\end{equation}
	\end{enumerate}
	The quadruple $(M,\omega,J,\nabla)$ will be called a \emph{special K\"ahler manifold}.
\end{definition}
A flat and torsion free connection $\nabla$ will be called a \emph{flat affine connection}. The fact that $\nabla$ is torsion free implies that identity \eqref{Eq1} is equivalent to require
\begin{equation}\label{Eq2}
J[X,Y]=\nabla_X(JY)-\nabla_Y(JX),\qquad X,Y\in\mathfrak{X}(M).
\end{equation}

Special cases of these kind of structures are pseudo-K\"ahler manifolds for which the Levi--Civita connection $\nabla$ associated to $k$ is flat. The identity $\iota\iota\iota.$ trivially holds because $\nabla J=0$. The converse is also true in the sense that if $\nabla$ is a flat affine symplectic connection and $\nabla J=0$, then $\nabla$ is the Levi-Civita connection associated to $k$ and $M$ is locally isometric to $\mathbb{C}^n$. These kind of manifolds are called \emph{flat special K\"ahler manifolds} and they have been characterized in \cite{BC1}. It is important to notice that there are no non-flat complete special K\"ahler manifolds, that is, if $k$ is complete then the Levi--Civita connection associated to it is flat; see \cite{L}. 

The notion of special K\"ahler manifold initially appeared in physics and it has its origins in certain supersymmetric field theories \cite{dWVP}. More specifically, affine special K\"ahler manifolds are exactly the allowed targets for the scalars of the vector multiplets of field theories with $N = 2$ rigid supersymmetry on 4-dimensional Minkowski space-time. Also, there exists a tight mathematical relationship between special real manifolds, which come from Hessian geometry and were introduced in \cite{AC}, and special K\"ahler manifolds. This relationship is given through the intrinsic description of an $r$-map which, on the physics side, corresponds to the dimensional reduction of rigid vector multiplets from 5 to 4 space-time dimensions; see \cite{AC}. Two important features of special K\"ahler manifolds are that their cotangent bundle carries the structure of a hyper-K\"ahler manifold and they are base of algebraic completely integrable systems; see \cite{F}. Some properties and the review of several interesting applications in physics and mathematics where special K\"ahler manifolds appear can be found in \cite{C}.

In this paper we mainly focus in giving three methods for constructing left invariant special K\"ahler structures on simply connected Lie groups. These kind of geometric objects will be called \emph{special K\"ahler Lie groups}. Accordingly, the infinitesimal objects associated to these kind of Lie groups will be called a \emph{special K\"ahler Lie algebras}.

Firstly, we get left invariant special K\"ahler structures on the cotangent bundle of a simply connected flat pseudo-Riemannian Lie group verifying the condition $\nabla^H J=0$. Here $\nabla^H$ denotes the Hess connection associated to the natural left invariant bi-Lagrangian transverse foliations that admits the cotangent bundle of any simply connected flat affine Lie group. We exhibit the conditions needed to ensure when this connection is geodesically complete. The collection of groups obtained by using this method are examples of flat special K\"ahler manifolds.

Secondly, we introduce the twisted cartesian product of two special special K\"ahler Lie algebras according to two Lie algebra representations by infinitesimal K\"ahler transformations. This method allows us to obtain examples of non-trivial left invariant special K\"ahler structures  in every even dimension $\geq 4$ since the only example in dimension $2$ is the trivial one, namely, $(\mathbb{R}^2,\omega_0,J_0,\nabla^0)$. We prove that every special K\"ahler algebra that admits a complex and non-degenerate left ideal can be obtained as the twisted cartesian product of two natural special K\"ahler Lie subalgebras.

Thirdly, we give a double extension process of a special K\"ahler Lie algebra via a real line and according to an infinitesimal linear symplectomorphism which defines a derivation of a left symmetric algebra and commutes with the complex structure. This double extension process gives us all simply connected special K\"ahler Lie groups with bi-invariant symplectic connections. Moreover, all Lie groups constructed as a double extension can be identified with a subgroup of symplectic (or K\"ahler) affine transformations of its Lie algebra containing a nontrivial  $1$-parameter subgroup formed by central translations. We end the paper by exhibiting a $1$-dimensional family of left invariant special K\"ahler structures parametrized by $\mathbb{R}$ in dimension $6$ with associated metric having signature $(4,2)$ and verifying $\nabla J\neq 0$.

We also show a characterization of left invariant flat special K\"ahler structures, give several non-trivial examples, and show some immediate consequences of the constructions mentioned above.
\section{Left invariant special K\"ahler structures and some examples}
For a general understanding of the basic concepts about symplectic Lie groups, left symmetric algebras, special K\"ahler manifolds, and their related topics, the reader is recommended to visit for instance the references \cite{BC2,Ba,Bu,C,K,M,V}. Let us assume that $M=G$ is a connected Lie group with Lie algebra $\mathfrak{g}$ and that all the geometric object that we are dealing with are left invariant, that is, the left multiplications in $G$ determine: symplectomorphisms of $(G,\omega)$, affine transformations of $(G,\nabla)$, and holomorphic maps of $(G,J)$.
\begin{definition}
A connected Lie group $G$ is called a \emph{special K\"ahler Lie group} if it can be equipped with a special K\"ahler structure $(\omega,J,\nabla)$ where $\omega$, $J$, and $\nabla$ are left invariant.
\end{definition}
 
Note that the pseudo-Riemannian metric $k$ on $G$ induced by $(\omega,J)$ is necessarily left invariant. The infinitesimal object associated to a special K\"ahler Lie group is the following:
\begin{definition}\label{definitionLie}
A real finite dimensional Lie algebra $\mathfrak{g}$ is called a \emph{special K\"ahler Lie algebra} if it can be equipped with a triple $(\omega,j,\cdot)$ where $\omega\in\wedge^2\mathfrak{g}^\ast$ is a non-degenerate scalar 2-cocycle, $j:\mathfrak{g}\to \mathfrak{g}$ is an integrable complex structure, and $\cdot:\mathfrak{g}\times \mathfrak{g}\to \mathfrak{g}$ is a left symmetric product such that for all $x,y,z\in \mathfrak{g}$ we have:
\begin{enumerate}
\item[$\iota.$] $(\mathfrak{g},\omega,j)$ is a pseudo-K\"ahler Lie algebra,
\item[$\iota\iota.$] $[x,y]=x\cdot y-y\cdot x$,
\item[$\iota\iota\iota.$] $\omega(x\cdot y,z)+\omega(y,x\cdot z)=0$, and
\item[$\iota\nu.$] $j\in Z^1_L(\mathfrak{g},\mathfrak{g})$.
\end{enumerate}
Here $Z^1_L(\mathfrak{g},\mathfrak{g})$ denotes the space of Lie algebra 1-cocycles with respect to the linear representation $L:\mathfrak{g}\to \mathfrak{gl}(\mathfrak{g})$ defined by $L_x(y):=x\cdot y$.
\end{definition}

For the purpose of this paper will be important to have the following formulas and definitions in mind:
\begin{itemize}
\item $\omega$ is a scalar 2-cocycle if it verifies the formula
$$\oint\omega([x,y],z)=\omega([x,y],z)+\omega([y,z],x)+\omega([z,x],y)=0,\qquad x,y,z\in \mathfrak{g}.$$
\item $\cdot:\mathfrak{g}\times \mathfrak{g}\to \mathfrak{g}$ is a left symmetric product on $\mathfrak{g}$ if it satisfies
$$x\cdot(y\cdot z)-(x\cdot y)\cdot z = y\cdot(x\cdot z)-(y\cdot x)\cdot z,\qquad x,y,z\in\mathfrak{g}.$$
If $(x,y,z):=x\cdot(y\cdot z)-(x\cdot y)\cdot z$ is the associator of $x$, $y$, and $z$ in $\mathfrak{g}$, the last identity means that $(x,y,z)=(y,x,z)$.
\item $j\in Z^1_L(\mathfrak{g},\mathfrak{g})$ if
$$j([x,y])=L_x(j(y))-L_y(j(x)),\qquad x,y\in\mathfrak{g}.$$
\item As $(\omega,j)$ defines a pseudo-K\"ahler structure on $\mathfrak{g}$ the formula
$$k(x,y)=\omega(x,j(y)),\qquad x,y\in\mathfrak{g}$$
defines a non-degenerate symmetric bilinear form on $\mathfrak{g}$, that is, a scalar product.
\end{itemize}
\begin{remark}
The data $(\mathfrak{g},\omega,\cdot)$ verifying the formulas $\iota\iota.$ and $\iota\iota\iota.$ from Definition \ref{definitionLie} is called a \emph{flat affine symplectic Lie algebra}. We will use the constructions introduced in \cite{Au,V} of these kind of objects for getting examples of left invariant special K\"ahler structures.
\end{remark}

The following result is clear:
\begin{lemma}
There exists a bijective correspondence between simply connected special K\"ahler Lie groups and special K\"ahler Lie algebras.
\end{lemma}

The model space of special K\"ahler Lie group is $((\mathbb{R}^{2n},+),\omega_0,J_0,\nabla^0)$ where $\omega_0$, $J_0$, and $\nabla^0$ are the canonical symplectic form, complex structure, and covariant derivative on $\mathbb{R}^{2n}$, respectively.
\begin{lemma}\label{dimension}
If $(G,\omega,J,\nabla)$ is a non-Abelian special K\"ahler Lie group, then $\textnormal{dim}G\geq 4$.
\end{lemma}
The previous Lemma is consequence of the following examples. Let $x^+$ denote the left invariant vector field associated to an element $x\in\mathfrak{g}$. The classification of left invariant flat affine symplectic connections in dimension $2$ can be found in \cite{An}.
\begin{example}
Aside $\nabla^0$, up to isomorphism there is another left invariant flat affine symplectic connection on $((\mathbb{R}^{2},+),\omega_0)$ and it is given by
$$\nabla_{e_1^+}e_1^+=e_2^+\qquad\textnormal{and}\qquad \nabla_{e_1^+}e_2^+=\nabla_{e_2^+}e_1^+=\nabla_{e_2^+}e_2^+=0.$$
It is easy to see that $(\nabla_{e_1^+}J_0)e_2^+\neq (\nabla_{e_2^+}J_0)e_1^+$. Thus $(\omega_0,J_0,\nabla)$ does not define a left invariant structure of special K\"ahler Lie group on $\mathbb{R}^2$.
\end{example}

\begin{example}
Let $G=\textnormal{Aff}(\mathbb{R})_0$ denote the connected component of the identity of the group of affine transformations of the real line with its natural left invariant symplectic form $\omega=\dfrac{1}{x^2}dx\wedge dy$. Up to isomorphism, there are two left invariant flat affine symplectic connections on $(\textnormal{Aff}(\mathbb{R})_0,\omega)$ and these are given by
\begin{enumerate}
\item[$\iota.$] $\nabla_{e_1^+}e_1^+=-e_1^+\qquad \nabla_{e_1^+}e_2^+=e_2^+\qquad\textnormal{and}\qquad \nabla_{e_2^+}e_1^+=\nabla_{e_2^+}e_2^+=0$, and
\item[$\iota\iota.$] $\overline{\nabla}_{e_1^+}e_1^+=-\dfrac{1}{2}e_1^+\qquad \overline{\nabla}_{e_1^+}e_2^+=\dfrac{1}{2} e_2^+\qquad \overline{\nabla}_{e_2^+}e_1^+=-\dfrac{1}{2} e_2^+ \qquad\textnormal{and}\qquad\overline{\nabla}_{e_2^+}e_2^+=0$.
\end{enumerate}
Here $e_1^+=x\dfrac{\partial}{\partial x}$ and $e_2^+=x\dfrac{\partial}{\partial y}$ determine a basis for the left invariant vector fields on $\textnormal{Aff}(\mathbb{R})_0$. The natural left invariant complex structure on $\textnormal{Aff}(\mathbb{R})_0$ is defined as $J(e_1^+)=e_2^+$. This is clearly integrable and moreover
$$(\nabla_{e_1^+}J)e_2^+\neq (\nabla_{e_2^+}J)e_1^+ \qquad\textnormal{and}\qquad (\overline{\nabla}_{e_1^+}J)e_2^+\neq (\overline{\nabla}_{e_2^+}J)e_1^+.$$
Therefore, $(\omega,J,\nabla)$ and $(\omega,J,\overline{\nabla})$ do not define left invariant special K\"ahler structures on $\textnormal{Aff}(\mathbb{R})_0$.
\end{example}
The following are three positive examples of left invariant special K\"ahler structures:
\begin{example}[Dimension 4]\label{ExampleD4}
Consider the Lie group $G_1=\mathbb{R}\ltimes_\rho \mathbb{R}^3$ determined by the semi-direct product of $(\mathbb{R},+)$ with $(\mathbb{R}^3,+)$ by means of the Lie group homomorphism $\rho:\mathbb{R}\to \textnormal{GL}(\mathbb{R}^3)$ defined by
$$\rho(t)=\left(\begin{array}{ccc}
e^t & 0 &0\\
0 & e^t & 0 \\
0 & 0 & e^{-t}
\end{array}%
\right).$$
The product in $G_1$ is explicitly given as 
$$(t,x,y,z)\cdot(t',x',y',z')=(t+t',e^tx'+x,e^ty'+y,e^{-t}z'+z).$$
A basis for the left invariant vector fields on $G_1$ is formed by
$$e_1^+=\dfrac{\partial}{\partial t},\qquad  e_2^+=e^t\dfrac{\partial}{\partial x},\qquad  e_3^+=e^t\dfrac{\partial}{\partial y},\qquad  e_4^+=e^{-t}\dfrac{\partial}{\partial z}.$$
Thus, the Lie algebra of $G_1$ is isomorphic to the vector space $\mathfrak{g}_1\cong \textnormal{Vect}_\mathbb{R}\lbrace e_1,e_2,e_3,e_4\rbrace$ with nonzero Lie brackets
$$[e_1,e_2]=e_2,\qquad [e_1,e_3]=e_3,\qquad\textnormal{and}\qquad[e_1,e_4]=-e_4.$$
The following data $(\omega,J,\nabla)$ defines a structure of special K\"ahler Lie group on $G_1$:
\begin{enumerate}
\item[$\iota.$] left invariant symplectic form $\omega=e^{-t}dy\wedge dt+dz\wedge dx$,
\item[$\iota\iota.$] left invariant complex structure $J(e_3^+)=e_2^+$ and $J(e_4^+)=e_1^+$, and
\item[$\iota\iota\iota.$] left invariant flat affine symplectic connection 
$$\nabla_{e_1^+}e_1^+=-e_1^+\qquad \nabla_{e_1^+}e_2^+=e_2^+\qquad \nabla_{e_1^+}e_3^+=e_3^+\qquad \nabla_{e_1^+}e_4^+=-e_4^+,\qquad\textnormal{and} $$
$$\nabla_{e_2^+}=\nabla_{e_3^+}=\nabla_{e_4^+}=0.$$
\end{enumerate}
\begin{remark}
The signature of the pseudo-Riemannian metric determined by $(\omega,J)$ is $(2,2)$ and $\textnormal{Aff}(\mathbb{R})_0$ can be identified with a Lagrangian Lie subgroup in $(G_1,\omega)$. Moreover, $\nabla J=0$.
\end{remark}
\end{example}

\begin{example}[Dimension 6]\label{ExampleD6}
Let us now consider the Lie group $G_2=\mathbb{R}\ltimes_\rho \mathbb{R}^5$ determined by the semi-direct product of $(\mathbb{R},+)$ with $(\mathbb{R}^5,+)$ by means of the Lie group homomorphism $\rho:\mathbb{R}\to \textnormal{GL}(\mathbb{R}^5)$ defined by
$$\rho(t)=\left(\begin{array}{ccccc}
1 & 0 &0 & 0 & 0\\
-t & 1 &0 & 0 & 0\\
0 & 0 & 1 & 0 & t\\
0 & 0 &-t & 1 & -t^2/2\\
0 & 0 &0 & 0 & 1
\end{array}%
\right).$$
The product in $G_2$ is explicitly given as 
$$(t,x,y,z,u,v)\cdot(t',x',y',z',u',v')=\left(t+t',x'+x,-tx'+y'+y,z'+tv'+z,-tz'+u'-t^2/2v'+u,v+v'\right).$$
A basis for the left invariant vector fields on $G_2$ is formed by
$$e_1^+=\dfrac{\partial}{\partial x}-t\dfrac{\partial}{\partial y},\qquad  e_2^+=\dfrac{\partial}{\partial t},\qquad  e_3^+=\dfrac{\partial}{\partial y},\qquad  e_4^+=\dfrac{\partial}{\partial z}-t\dfrac{\partial}{\partial u},$$
$$e_5^+=\dfrac{\partial}{\partial u},\qquad e_6^+=t\dfrac{\partial}{\partial z}-t^2/2\dfrac{\partial}{\partial u}+\dfrac{\partial}{\partial v}.$$
Therefore, the Lie algebra of $G_2$ is isomorphic to the vector space $\mathfrak{g}_2\cong \textnormal{Vect}_\mathbb{R}\lbrace e_1,e_2,e_3,e_4,e_5,e_6\rbrace$ with nonzero Lie brackets
$$[e_1,e_2]=e_3,\qquad [e_2,e_4]=-e_5,\qquad\textnormal{and}\qquad[e_2,e_6]=e_4.$$
The following data $(\omega,J,\nabla)$ defines a structure of special K\"ahler structure on $G_2$:
\begin{enumerate}
	\item[$\iota.$] left invariant symplectic form $\omega=tdz\wedge dt +du\wedge dt+t^2/2dt\wedge dv+dz\wedge dx+dv\wedge dy$,
	\item[$\iota\iota.$] left invariant complex structure $J(e_4^+)=e_1^+$, $J(e_5^+)=e_3^+$ and $J(e_6^+)=e_2^+$, and 
	\item[$\iota\iota\iota.$] left invariant flat affine symplectic connection 
	$$\nabla_{e_2^+}e_1^+=-e_3^+\qquad \nabla_{e_2^+}e_2^+=e_1^+\qquad \nabla_{e_2^+}e_4^+=-e_5^+\qquad \nabla_{e_2^+}e_6^+=e_4^+,\qquad \nabla_{e_2^+}e_3^+=\nabla_{e_2^+}e_5^+=0$$
	$$\textnormal{and}\qquad\nabla_{e_1^+}=\nabla_{e_3^+}=\nabla_{e_4^+}=\nabla_{e_5^+}=\nabla_{e_6^+}=0.$$
\end{enumerate}
\begin{remark}
The signature of the pseudo-Riemmanian metric determined by $(\omega,J)$ is $(2,4)$ and the 3-dimensional Heisenberg group $H_3$ can be identified with a Lagrangian Lie subgroup in $(G_2,\omega)$. Moreover, $\nabla J=0$.
\end{remark}
\end{example}

The groups constructed in Examples \ref{ExampleD4} and \ref{ExampleD6} are examples of flat special K\"ahler manifolds which have been characterized in \cite{BC1}. We got them using the method to be introduced in next section.
 
The following is an example of special K\"ahler Lie group in dimension $4$ for which $\nabla J\neq 0$. It will be important for getting other interesting examples. 
\begin{example}[Dimension $4$ with $\nabla J\neq 0$]\label{KeyExample1}
Consider the Lie group $G_3=\mathbb{R}^2\ltimes_\rho \mathbb{R}^2$ determined by the semi-direct product of $(\mathbb{R}^2,+)$ with $(\mathbb{R}^2,+)$ by means of the Lie group homomorphism $\rho:\mathbb{R}^2\to \textnormal{GL}(\mathbb{R}^2)$ defined by
$$\rho(t,s)=\left(\begin{array}{cc}
t+s+1 & t+s \\
-(t+s) & -(t+s)+1
\end{array}%
\right).$$
The product in $G_3$ is explicitly given as 
$$(t,s,x,y)\cdot(t',s',x',y')=(t+t',s+s',(t+s+1)x'+(t+s)y'+x,-(t+s)x'+(1-(t+s))y'+y).$$
A basis for the left invariant vector fields on $G_3$ is formed by
$$e_1^+=\dfrac{\partial}{\partial t},\qquad  e_2^+=(t+s+1)\dfrac{\partial}{\partial x}-(t+s)\dfrac{\partial}{\partial y}$$
$$e_3^+=\dfrac{\partial}{\partial s},\qquad  e_4^+=(t+s)\dfrac{\partial}{\partial x}+(1-(t+s))\dfrac{\partial}{\partial y}.$$
Therefore, the Lie algebra of $G_3$ is isomorphic to the vector space $\mathfrak{g}_3\cong \textnormal{Vect}_\mathbb{R}\lbrace e_1,e_2,e_3,e_4\rbrace$ with nonzero Lie brackets
$$[e_1,e_2]=[e_1,e_4]=[e_3,e_2]=[e_3,e_4]=e_2-e_4.$$
The following data $(\omega,J,\nabla)$ defines a structure of special K\"ahler structure on $G_3$:
\begin{enumerate}
	\item[$\iota.$] left invariant symplectic form $\omega=(1-(t+s))dt\wedge dx+(t+s)(dy\wedge dt+dx\wedge ds)+(1+t+s)dy\wedge ds$,
	\item[$\iota\iota.$] left invariant complex structure $J(e_1^+)=e_2^+$ and $J(e_3^+)=e_4^+$, and
	\item[$\iota\iota\iota.$] left invariant flat affine symplectic connection 
	\begin{center}
	\begin{tabular}{c|c|c|c|c}
		$\nabla$	& $e_1^+$ & $e_2^+$ & $e_3^+$ & $e_4^+$ \\
		\hline
		$e_1^+$	& $-e_1^++e_3^+$ & $e_2^+-e_4^+$ & $-e_1^++e_3^+$ & $e_2^+-e_4^+$ \\
		\hline
		$e_2^+$	& $0$ & $2e_1^+-2e_3^+$  & $0$ & $2e_1^+-2e_3^+$ \\
		\hline
		$e_3^+$	& $-e_1^++e_3^+$ & $e_2^+-e_4^+$ & $-e_1^++e_3^+$ & $e_2^+-e_4^+$ \\
		\hline
		$e_4^+$	& $0$ & $2e_1^+-2e_3^+$  & $0$ & $2e_1^+-2e_3^+$ \\
	
	\end{tabular}
	\end{center}
\end{enumerate}
\begin{remark}
	The signature of the pseudo-Riemmanian metric determined by $(\omega,J)$ is $(2,2)$. Moreover, given that $\nabla_{e_1^+}\circ J\neq J\circ \nabla_{e_1^+}$, we obtain that $\nabla J\neq 0$. 
\end{remark}
\end{example}

\subsection{\'Etale K\"ahler affine representations and the case $\nabla J=0$}
A finite dimensional real vector space $(V,\omega,J)$ is called a \emph{K\"ahler vector space} if $(V,\omega)$ is a symplectic vector space and $J:V\to V$ is a linear complex structure on $V$ such that
\begin{enumerate}
\item[$\iota.$] $\omega(J(x),J(y))=\omega(x,y)$, and
\item[$\iota\iota.$] $k(x,y)=\omega(x,J(y))$ is a scalar product on $V$.
\end{enumerate}

Let $\textnormal{Sp}(V,\omega)$ and $\textnormal{GL}(V,J)$ denote the groups of linear symplectomorphisms and linear complex transformations of $(V,\omega)$ and $(V,J)$, respectively. We define the group of \emph{linear K\"ahler transformations} of $(V,\omega,J)$ as
$$\textnormal{KL}(V,\omega,J):=\textnormal{Sp}(V,\omega)\cap \textnormal{GL}(V,J).$$
This is a Lie group with Lie algebra $\mathfrak{kl}(V,,\omega,J)=\mathfrak{sp}(V,\omega)\cap \mathfrak{gl}(V,J)$ where $\mathfrak{sp}(V,\omega)$ and $\mathfrak{gl}(V,J)$ are the Lie algebras of $\textnormal{Sp}(V,\omega)$ and $\textnormal{GL}(V,J)$, respectively. More precisely, $\mathfrak{kl}(V,,\omega,J)$ is composed by elements $A\in \mathfrak{gl}(V)$ verifying both
$$\omega(A(x),y)+\omega(x,A(y))=0\qquad \textnormal{and}\qquad AJ=JA.$$
The group of \emph{K\"ahler affine transformations} of $(V,\omega,J)$ is defined as the semi-direct product $V\rtimes_{\textnormal{Id}}\textnormal{KL}(V,\omega,J)$ with respect to the identity representation $\textnormal{Id}:\textnormal{KL}(V,\omega,J)\hookrightarrow \textnormal{GL}(V)$. 

Motivated by the definition of \'etale affine representations given in \cite{K} and \cite{M}, we set up the following definition:
\begin{definition}
An \emph{\'etale K\"ahler affine representation} of a Lie group $G$ is a Lie group homomorphism $\rho:G\to V\rtimes_{\textnormal{Id}}\textnormal{KL}(V,\omega,J)$ such that the left action of $G$ on $V$ defined by $g\cdot v:=\rho(g)(v)$ for all $g\in G$ and $v\in V$, admits a point of open orbit and discrete isotropy.
\end{definition}
As the Lie algebras of a connected Lie group $G$ and its universal covering Lie group $\widetilde{G}$ are isomorphic, from now on we assume that $G$ is simply connected. The following is a characterization of left invariant flat special K\"ahler structures.
\begin{theorem}\label{etale}
Let $G$ be a simply connected Lie group with Lie algebra $\mathfrak{g}$. Then $G$ can be equipped with a structure of special K\"ahler Lie group $(\omega,J,\nabla)$ with $\nabla J=0$ if and only if it admits an \'etale K\"ahler affine representation by a K\"ahler vector space $(V,\omega,J)$.
\end{theorem}
\begin{proof}
Suppose that $(G,\omega^+,J,\nabla)$ is a simply connected special K\"ahler Lie group such that $\nabla J=0$. Let $(\mathfrak{g},\omega,j,\cdot)$ be the special K\"ahler Lie algebra associated to $(G,\omega^+,J,\nabla)$. Namely, if $e$ denotes the identity of $G$, we have that $\omega:=\omega^+_e$, $j:=J_e$, and $x\cdot y=L_x(y):= (\nabla_{x^+}y^+)(e)$. Given that asking for $\nabla J=0$ is equivalent to require that  $L_x\circ j=j\circ L_x$ for all $x\in\mathfrak{g}$, it follows that the map $\theta:\mathfrak{g} \to \mathfrak{g}\rtimes_{\textnormal{id}}\mathfrak{kl}(\mathfrak{g},\omega,j)$ defined by $\theta(x):=(x,L_x)$, is a well defined Lie algebra homomorphism. Thus, passing by the exponential of $G$, we get an \'etale K\"ahler affine representation $\rho\colon G \to \mathfrak{g}\rtimes_{\textnormal{Id}}\textnormal{KL}(\mathfrak{g},\omega,j)$ determined by
$$\rho(\exp_G(x))=\left(\sum_{m=1}^\infty \dfrac{1}{m!}(L_x)^{m-1}(x),\sum_{m=0}^\infty \dfrac{1}{m!}(L_x)^m\right),$$
for which clearly $0\in\mathfrak{g}$ is a point of open orbit and discrete isotropy.
	
Conversely, let $\rho:G\to V\rtimes_{\textnormal{Id}}\textnormal{KL}(V,\omega,J)$, defined as $\rho(g):=(Q(g),F_g)$ for all $g\in G$, be an \'etale K\"ahler affine representation of $G$. Let $v\in V$ be a point of open orbit and discrete isotropy. This implies that the orbital map $\pi\colon G\to \textnormal{Orb}(v)$ defined by 
$g\mapsto Q(g)+F_g(v)$ is a local diffeomorphism. Differentiating at the identity of $G$, 
we obtain a Lie algebra homomorphism $\theta\colon \mathfrak{g} \to V\rtimes_{\textnormal{id}}\mathfrak{kl}(V,\omega,J)$ given by $x \mapsto (q(x),f_x)$, where the linear map  $\psi_v\colon \mathfrak{g} \to V$ defined by $x \mapsto q(x)+f_x(v)$ is a linear isomorphism; see \cite{M}. Moreover, the map $f:\mathfrak{g}\to \mathfrak{kl}(V,\omega,J)$ is also a Lie algebra homomorphism and $q:\mathfrak{g}\to V$ is a Lie algebra 1-cocycle with respect to the linear representation $f$.
Let us now define the following objects on $\mathfrak{g}$:
\begin{enumerate}
\item[$\iota.$] the skew-symmetric bilinear form $\widetilde{\omega}(x,y):=\omega(\psi_v(x),\psi_v(y))$,
\item[$\iota\iota.$] the product $x\cdot y=L_x(y):=(\psi_v^{-1}\circ f_x\circ \psi_v)(y)$, and
\item[$\iota\iota\iota.$] the complex structure $j(x):=(\psi_v^{-1}\circ J\circ \psi_v)(x)$.
\end{enumerate}

The fact that $q$ is a Lie algebra 1-cocycle with respect to $f$ implies that $[x,y]=x\cdot y - y\cdot x$ for all $x,y\in \mathfrak{g}$. As $f$ is a linear representation and $f_x\in \mathfrak{kl}(V,\omega,J)$, we get that $L_{[x,y]}=[L_x,L_y]$. For last two reasons, it follows that $\cdot$ must be a left symmetric product on $\mathfrak{g}$.

Clearly $\widetilde{\omega}$ is non-degenerate. Furthermore, the fact that $f_x\in \mathfrak{kl}(V,\omega,J)$ implies 
\begin{equation}\label{3}
\widetilde{\omega}(L_x(y),z)+\widetilde{\omega}(y,L_x(z))=0,\qquad x,y,z\in\mathfrak{g}.
\end{equation}
Identity \eqref{3} and the fact that $[x,y]=L_x(y)-L_y(x)$ implies that $\omega$ is a 2-cocycle; visit \cite{V}. 

Finally, a straightforward computation allows us to conclude that the fact that $f_x\circ J= J\circ f_x$ for all $x\in \mathfrak{g}$ implies that 
$$[j(x),j(y)]-[x,y]=j[j(x),y]+j[x,j(y)]\qquad \textnormal{and}\qquad (j\circ L_x)(y)=(L_x\circ j)(y)$$
for all $x,y\in \mathfrak{g}$. That is, $j$ is integrable and it satisfies $[j,L_x]=0$ for all $x\in\mathfrak{g}$. 

Note that
$$\widetilde{k}(x,y)=\widetilde{\omega}(x,j(y))=\omega(\psi_v(x), (J\circ \psi_v)(y))=k(\psi_v(x),\psi_v(y)),$$
defines a scalar product on $\mathfrak{g}$ whose signature agrees with the signature of $k$. Hence, the quadruple $(\mathfrak{g},\widetilde{\omega},j,\cdot)$ is a special K\"ahler Lie algebra for which we may induce a left invariant special K\"ahler structure $(\omega,J,\nabla)$ on $G$ such that $\nabla J=0$.
\end{proof}
\subsection{Hessian property}
Let $(M,\nabla)$ be a flat affine manifold. A pseudo-Riemannian metric $k$ on $M$ is said to be \emph{Hessian} with respect to $\nabla$ if using the affine coordinates of $M$ induced by $\nabla$ it can be locally written as $k=\nabla^2\varphi$ where $\varphi$ is a local smooth function. That is,
$$k=\sum_{i,j}\dfrac{\partial^2 \varphi}{\partial x_i \partial x_j}dx^i\otimes dx^j,$$
where $(x_1,\cdots,x_n)$ is a system of local affine coordinates of $M$ induced by $\nabla$; see \cite{SY} for further details. It is well known that the pseudo-metric of a special K\"ahler manifold is indeed Hessian; compare \cite{F}. We will use a different approach to prove this fact for the left invariant case.

When $M=G$ is a connected Lie group and both $\nabla$ and $k$ are left invariant, the triple $(G,\nabla,k)$ is called a \emph{pseudo-Hessian Lie group}. The infinitesimal object associated to a pseudo-Hessian Lie group is the following:
\begin{definition}\cite{S}
Let $(\mathfrak{g},\cdot)$ be a finite dimensional left symmetric algebra. A scalar product $k$ on $\mathfrak{g}$ is said to be \emph{left symmetric} if it verifies
\begin{equation}\label{HessianCondition}
k(x\cdot y-y\cdot x,z)=k(x,y\cdot z)-k(y,x\cdot z),\qquad x,y,z\in\mathfrak{g}.
\end{equation}
Moreover, if $\mathfrak{g}$ is a Lie algebra whose Lie bracket is given by the commutator of $\cdot$ then the triple $(\mathfrak{g},\cdot,k)$ is called a \emph{pseudo-Hessian Lie algebra}.
\end{definition}

\begin{lemma}\cite{S}
There exists a bijective correspondence between simply connected pseudo-Hessian Lie groups and pseudo-Hessian Lie algebras.
\end{lemma}
As a consequence of \cite[Lemma. 4.1]{V} we get the following. Let $(G,\omega^+,J^+,\nabla)$ be a special K\"ahler Lie group and let $(\mathfrak{g},\omega,j,\cdot)$ be its respective special K\"ahler Lie algebra. If $k$ is the scalar product on $\mathfrak{g}$ induced by $(\omega,j)$, then we have the relation
$$k(x,y)=\omega(x,j(y)),\qquad x,y\in\mathfrak{g}.$$
Recall that $j\in Z_L^1(\mathfrak{g},\mathfrak{g})$. Thus, for all $x,y,z\in\mathfrak{g}$ we obtain
\begin{eqnarray*}
	k(x\cdot y-y\cdot x,z) & = & k([x,y],z)\\
	& = &  \omega([x,y],j(z))\\
	& = & -\omega(j([x,y]),z)\\
	& = & -\omega(L_x(j(y))-L_y(j(x)),z)\\
	& = & -\omega(L_x(j(y)),z)+\omega(L_y(j(x)),z)\\
	& = & \omega(j(y),x\cdot z)-\omega(j(x),y\cdot z)\\
	& = & \omega(x,j(y\cdot z))-\omega(y,j(x\cdot z))\\
	& = & k(x,y\cdot z)-k(y,x\cdot z).
\end{eqnarray*}
So, the triple $(\mathfrak{g},\cdot,k)$ defines a pseudo-Hessian Lie algebra and hence the left invariant pseudo-Riemannian metric $k^+$ induced by $(\omega^+,J^+)$ is a Hessian pseudo-metric on $G$.

\begin{remark}
The Hessian pseudo-metric $k^+$ allows us to define a structure of pseudo-K\"ahler Lie group on the cotangent bundle of a special K\"ahler Lie group. This left invariant pseudo-K\"ahler structure will be described in the next section; see Lie bracket \eqref{Hess1}, symplectic form \eqref{Hess2}, and complex structure of item $\iota\iota\iota.$ The integrability of such a complex structure follows from the fact that $k$ is a left symmetric scalar product, that is, it verifies identity \eqref{HessianCondition}.
\end{remark}

\section{Cotangent bundle of a flat pseudo-Riemannian Lie group}
The aim of this section is to give a method for constructing left invariant special K\"ahler structures on the cotangent bundle of a simply connected flat pseudo-Riemannian Lie group. We will use the construction of left invariant flat affine symplectic structures introduced in \cite{V} (see also \cite{NB}) starting from a connected flat affine Lie group.  The collection of groups constructed here are examples of flat special K\"ahler manifolds. See for instance \cite{BC1} to know more about it. 

Let us consider the triple $(G,k^+,\nabla)$ where $(G,\nabla)$ is a flat affine Lie group ($\nabla$ is a left invariant flat affine connection) and $(G,k^+)$ is a pseudo-Riemannian Lie group ($k^+$ is a left invariant pseudo-Riemannian metric). Let $\mathfrak{g}$ be the Lie algebra of $G$ and recall that from previous section we are assuming that $G$ is simply connected.

With the data $(G,\nabla)$ it is possible to construct the following objects. Let $L^\ast:\mathfrak{g}\to \mathfrak{gl}(\mathfrak{g}^*)$ denote the dual linear representation associated $L:\mathfrak{g}\to \mathfrak{gl}(\mathfrak{g})$ where $L_x(y)=x\cdot y:=(\nabla_{x^+}y^+)(e)$ for all $x,y\in\mathfrak{g}$. Passing to exponential, we get a Lie group homomorphism $F:G \to \textnormal{GL}(\mathfrak{g}^*)$ determined by 
$$F(\exp_G(x))=\sum_{m=0}^{\infty}\dfrac{1}{m!}(L_x^\ast)^m.$$
The cotangent bundle $T^\ast G\approx G\times \mathfrak{g}^\ast$ may be endowed with a Lie group structure given by
$$(g,\alpha)\cdot(g',\beta)=(gg',F(g)(\beta)+\alpha),\qquad g,g'\in G\quad \textnormal{and}\quad\alpha,\beta\in\mathfrak{g}^\ast,$$
whose Lie algebra is the vector space $\mathfrak{g}\oplus \mathfrak{g}^\ast$ with Lie bracket
\begin{equation}\label{Hess1}
[x+\alpha,y+\beta]_\ast=[x,y]+L^*_x(\beta)-L^*_y(\alpha),\qquad x,y\in\mathfrak{g}\quad \textnormal{and}\quad\alpha,\beta\in\mathfrak{g}^\ast.
\end{equation}
\begin{enumerate}
\item[$\iota.$] In \cite{MR} the authors proved that 
\begin{equation}\label{Hess2}
\omega(x+\alpha,y+\beta)=\alpha(y)-\beta(x),
\end{equation}
defines a scalar 2-cocycle on $(\mathfrak{g}\oplus \mathfrak{g}^\ast, [\cdot,\cdot]_\ast)$.
\end{enumerate}
As $\mathfrak{g}$ and $\mathfrak{g}^\ast$ are Lagrangian Lie subalgebras of $(\mathfrak{g}\oplus \mathfrak{g}^\ast, [\cdot,\cdot]_\ast,\omega)$, by the Frobenius Theorem, they induce two left invariant Lagrangian transverse foliations on $T^\ast G$. Associated to such foliations, there exists a unique torsion free symplectic connection $\nabla^H$ which parallelizes both foliations. In the literature $\nabla^H$ is known as the Hess connection or the canonical connection of a bi-Lagrangian manifold (see \cite{H} for more details and the definition of the Hess connection).
\begin{enumerate}
\item[$\iota\iota.$] In \cite{V} the author showed that the Hess connection $\nabla^H$ associated to such a left invariant Lagrangian transverse foliations is also left invariant, flat, and it can be explicitly determined by
\begin{equation}\label{HessConnection}
\nabla^H_{(x+\alpha)^+}(y+\beta)^+=(xy+L_x^\ast(\beta))^+\qquad x,y\in\mathfrak{g}\quad \textnormal{and}\quad \alpha,\beta \in\mathfrak{g}^\ast
\end{equation}
where $xy=(\nabla_{x^+}y^+)(e)$.
\end{enumerate}
With the data $(G,k^+)$ we can define the following complex structure on $\mathfrak{g}\oplus \mathfrak{g}^\ast$. If we denote by $k:=k^+_e$ the scalar product on $\mathfrak{g}$, then we have that $k^\flat:\mathfrak{g}\to \mathfrak{g}$ defined by $k^\flat(x)=k(x,\cdot)$ is a linear isomorphism. 
\begin{enumerate}
\item[$\iota\iota\iota.$] Define the complex structure $j$ on $\mathfrak{g}\oplus \mathfrak{g}^\ast$ as
$$
j(x+\alpha)=z_\alpha-k^\flat(x)\qquad x\in\mathfrak{g}\quad \textnormal{and}\quad \alpha\in\mathfrak{g}^\ast
$$
where $z_\alpha$ is the unique element in $\mathfrak{g}$ such that $\alpha=k^\flat(z_\alpha)$. This is integrable with respect to the Lie bracket \eqref{Hess1} if and only if
\begin{equation}\label{Left2-cocycle}
k([x,y],z)=k(x,L_y(z))-k(y,L_x(z)),\qquad x,y,z\in\mathfrak{g}.
\end{equation}
For the definition of such a complex structure see for instance \cite{DM}.
\end{enumerate}
Let $\cdot^H$ denote the left symmetric product on $\mathfrak{g}\oplus \mathfrak{g}^\ast$ determined by $\nabla^H$ and $L^H$ its associated linear representation. Note that, on the one hand
$$j([x+\alpha,y+\beta]_\ast)=-j(k(z_\beta,L_x(\cdot)))+j(k(z_\alpha,L_y(\cdot)))-k^\flat([x,y]),$$
and on the other hand,
$$L^H_{(x+\alpha)}(j(y+\beta))-L^H_{(y+\beta)}(j(x+\alpha))= L_x(z_\beta)-L_y(z_\alpha)+k(y,L_x(\cdot))-k(x,L_y(\cdot)).$$
Thus, $j\in Z^1_{L^H}(\mathfrak{g}\oplus \mathfrak{g}^\ast,\mathfrak{g}\oplus \mathfrak{g}^\ast)$ if and only if $k$ satisfies identity \eqref{Left2-cocycle} and 
%
\begin{equation}\label{flatness}
k(z_\beta,L_x(\cdot))-k(z_\alpha,L_y(\cdot))=-k^\flat(L_x(z_\beta))+k^\flat(L_y(z_\alpha)),
\end{equation}
for all $x,y\in\mathfrak{g}$ and $\alpha,\beta\in\mathfrak{g}^\ast$.
\begin{lemma}
Identities \eqref{Left2-cocycle} and \eqref{flatness} hold true if and only if $\nabla$ is the Levi--Civita connection associated to $(G,k^+)$.
\end{lemma}
\begin{proof}
Because the Levi--Civita connection associated to a left invariant pseudo-Riemannian metric trivially satisfies identity \eqref{Left2-cocycle}, it is enough to prove that identity	\eqref{flatness} holds true if and only if $\nabla$ is the Levi-Civita connection determined by $(G,k^+)$. If $\nabla$ is the Levi--Civita connection, then it follows that
$$k(L_x(y),z)+k(y,L_x(z))=0,\qquad x,y,z\in\mathfrak{g}.$$
It is easy to see that last identity implies \eqref{flatness}. Conversely, if identity \eqref{flatness} is true, then when $y=0$ and $z_\alpha=z_\beta$ we get that $k(z_\beta,L_x(\cdot))=-k(L_x(z_\beta),\cdot)$ for all $x\in \mathfrak{g}$ and $\beta\in \mathfrak{g}^\ast$. Thus, the fact that $k^\flat$ is a linear isomorphism implies that $\nabla$ must be the Levi-Civita connection associated to $(G,k^+)$.
\end{proof}
Summing up, we obtain the following result:
\begin{theorem}
The quadruple $(\mathfrak{g}\oplus \mathfrak{g}^\ast,\omega,j,\cdot^H)$ defines a special K\"ahler Lie algebra if and only if $(G,k^+)$ is a flat pseudo-Riemannian Lie group with Levi-Civita connection $\nabla$.
\end{theorem}
\begin{remark}
The pseudo-Riemannian metric $\widetilde{k}^+$ on $T^\ast G$ induced by the pair $(\omega,j)$ is determined by the scalar product
$$\widetilde{k}(x+\alpha,y+\beta)=k(x,y)+k(z_\alpha,z_\beta)\qquad x,y\in\mathfrak{g}\quad \textnormal{and}\quad \alpha,\beta \in\mathfrak{g}^\ast.$$
Therefore, if the signature of $k$ is $(p,q)$, then the signature of $\widetilde{k}$ is $(2p,2q)$. This may be easily verified by taking up an orthonormal basis of $(\mathfrak{g},k)$. In particular, if $(G,k^+)$ is a flat Riemannian Lie group, then $(T^\ast G,\omega,j,\nabla^H)$ is a special K\"ahler Lie group with $\widetilde{k}^+$ a left invariant Riemannian metric.
\end{remark}

\begin{corollary}
$\nabla^H j=0$.
\end{corollary}
\begin{proof}
This is a straightforward computation using the fact that $\nabla$ is the Levi--Civita connection associated to $k^+$.
\end{proof}

\begin{corollary}
The Hess connection $\nabla^H$ is geodesically complete if and only if $G$ is unimodular.
\end{corollary}
\begin{proof}
In \cite{V} was proved that $\nabla^H$ is geodesically complete if and only if $\nabla$ is geodesically complete. On the other hand, in \cite{AM} the authors showed that the Levi--Civita connection associated to a left invariant flat pseudo-Riemannian metric is geodesically complete if and only if $G$ is unimodular. So, the result follows.
\end{proof}
Using the previous simple construction and the constructions of flat pseudo-Riemannian Lie groups introduced in \cite{Au,AM} it is possible to get several interesting examples of flat special K\"ahler Lie groups. For instance:
\begin{example}
On the one hand, the special K\"ahler Lie group $G_1$ of Example \ref{ExampleD4} can be obtained from the 2-dimension flat Lorentzian Lie group $(\textnormal{Aff}(\mathbb{R})_0,k_1)$ where $k_1=\dfrac{1}{x^2}(dx\otimes dy+dy\otimes dx)$. Given that $\textnormal{Aff}(\mathbb{R})_0$ is not unimodular, the Hess connection $\nabla^H$ on $G_1$ is not geodesically complete.

On the other hand, the special K\"ahler Lie group $G_2$ of Example \ref{ExampleD6} can be obtained from the 3-dimension flat Lorentzian Lie group $(H_3,k_2)$ where $H_3$ is the 3-dimensional Heisenberg Lie group and $k_2=dx\otimes dz+dy\otimes dy+dz\otimes dx-x(dx\otimes dy+ dy\otimes dx)$. Given that $H_3$ is unimodular, the Hess connection $\nabla^H$ on $G_2$ is geodesically complete. The $(2n+1)$-dimensional Heisenberg group $H_{2n+1}$ does not admit a structure of flat pseudo-Riemannian Lie group for all $n\geq 2$; see \cite{AM}.
\end{example}

\section{Twisted cartesian products}

The twisted product of Lie groups (resp. Lie algebras) defined by means of linear representations is a well known construction which may be viewed as the generalization of the semidirect product of groups (resp. algebras); see for instance \cite{R}. In this section we explain a way for constructing a structure of special K\"ahler Lie algebra on the cartesian product of two special K\"ahler Lie algebras which is twisted by two linear representations by infinitesimal K\"ahler transformations of such Lie algebras. First steps of the construction that we will introduce below are motivated by a construction of flat pseudo-Riemannian metrics made in \cite{Au}. It is worth mentioning here that our twisted cartesian product can be obtained by a ``double reduction'' process that we shall explain at the end of the section. As we will see later, by using this construction we can obtain examples of non-trivial left invariant special K\"ahler structures in every even dimension $\geq 4$.

Recall that the group of linear K\"ahler transformations of K\"ahler vector space $(V,\omega,J)$ is defined as the Lie group $\textnormal{KL}(V,\omega,J):=\textnormal{Sp}(V,\omega)\cap \textnormal{GL}(V,J)$ whose Lie algebra is given by $\mathfrak{kl}(V,\omega,J)=\mathfrak{sp}(V,\omega)\cap \mathfrak{gl}(V,J)$.
\begin{Assumption}
	Let $(\mathfrak{g}_1,\omega_1,j_1,\cdot_1)$ and $(\mathfrak{g}_2,\omega_2,j_2,\cdot_2)$ be two special K\"ahler Lie algebras for which there exist two linear Lie algebra representations
	$$\theta:\mathfrak{g}_1\to \mathfrak{kl}(\mathfrak{g}_2,\omega_2,j_2)\qquad\textnormal{and}\qquad\rho:\mathfrak{g}_2\to \mathfrak{kl}(\mathfrak{g}_1,\omega_1,j_1).$$
\end{Assumption}
Our goal here is to define a structure of special K\"ahler Lie algebra on $\mathfrak{g}_1\oplus \mathfrak{g}_2$ using these two representations $(\theta,\rho)$ so that both $\mathfrak{g}_1$ and $\mathfrak{g}_2$ become special K\"ahler Lie subalgebras of it.
\begin{enumerate}
	\item[$\iota.$] Let us begin by defining our candidate of left symmetric product on $\mathfrak{g}_1\oplus \mathfrak{g}_2$. Let $L:\mathfrak{g}_1\to\mathfrak{gl}(\mathfrak{g}_1)$ and $L':\mathfrak{g}_2\to\mathfrak{gl}(\mathfrak{g}_2)$ denote the linear representations induced by the left symmetric products $\cdot_1$ and $\cdot_2$, respectively. We will use the linear transformations $(\theta,\rho)$ for defining the action of $\mathfrak{g}_1$ over $\mathfrak{g}_2$, and vice-versa, as $\theta(x_1)(x_2)=x_1\cdot x_2$ and $\rho(x_2)(x_1)=x_2\cdot x_1$ for all $x_1\in\mathfrak{g}_1$ and $x_2\in\mathfrak{g}_2$. Accordingly, we define the product $\cdot$ on $\mathfrak{g}_1\oplus \mathfrak{g}_2$ as
	\begin{equation}\label{TwistedProduct}
	(x_1+x_2)\cdot(y_1+y_2)=L_{x_1}(y_1)+\rho(x_2)(y_1)+\theta(x_1)(y_2)+L'_{x_2}(y_2).
	\end{equation}
	Let us now deduce the conditions needed for $\cdot$ to be a left symmetric product. We just need to look at the general expression for the associator $(x_1+x_2,y_1+y_2,z_1+z_2)$ of $\cdot$ in $\mathfrak{g}_1\oplus \mathfrak{g}_2$. Indeed, 
	\begin{eqnarray*}
		& & (x_1+x_2,y_1+y_2,z_1+z_2)\\
		& = & (x_1+x_2)\cdot((y_1+y_2)\cdot(z_1+z_2))-((x_1+x_2)\cdot(y_1+y_2))\cdot(z_1+z_2)\\
		& = & x_1\cdot_1 (y_1\cdot_1 z_1)-(x_1\cdot_1 y_1)\cdot_1 z_1+x_1\cdot_1\rho(y_2)(z_1)+\rho(x_2)(y_1\cdot_1 z_1)\\
		& + & \rho(x_2)(\rho(y_2)(z_1))-\rho(\theta(x_1)(y_2))(z_1)-\rho(x_2)(y_1)\cdot_1 z_1-\rho(x_2\cdot_2 y_2)(z_1)\\
		& + & x_2\cdot_2 (y_2\cdot_2 z_2)-(x_2\cdot_2 y_2)\cdot_2 z_2+\theta(x_1)(\theta(y_1)(z_2))+\theta(x_1)(y_2\cdot_2 z_2)\\
		& + & x_2\cdot_2 \theta(y_1)(z_2)-\theta(x_1\cdot_1 y_1)(z_2)-\theta(x_1)(y_2)\cdot_2 z_2-\theta(\rho(x_2)(y_1))(z_2).
	\end{eqnarray*}
	Recall that both $\rho$ and $\theta$ are Lie algebra representations. Moreover, $(x_1,y_1,z_1)=(y_1,x_1,z_1)$ and $(x_2,y_2,z_2)=(y_2,x_2,z_2)$ are verified in $\mathfrak{g}_1$ and $\mathfrak{g}_2$, respectively since $\cdot_1$ and $\cdot_2$ are left symmetric products. Therefore, the identity $(x_1+x_2,y_1+y_2,z_1+z_2)=(y_1+y_2,x_1+x_2,z_1+z_2)$ holds in $\mathfrak{g}_1\oplus \mathfrak{g}_2$ if and only if
	\begin{equation}\label{Twisted1}
	L_{x_1}\circ\rho(x_2)-\rho(x_2)\circ L_{x_1}=\rho(\theta(x_1)(x_2))-L_{\rho(x_2)(x_1)},\qquad\textnormal{and}
	\end{equation}
	\begin{equation}\label{Twisted2}
	L'_{x_2}\circ\theta(x_1)-\theta(x_1)\circ L'_{x_2}=\theta(\rho(x_2)(x_1))-L'_{\theta(x_1)(x_2)}
	\end{equation}
	for all $x_1\in\mathfrak{g}_1$ and $x_2\in\mathfrak{g}_2$.
	
	\item[$\iota\iota.$] The Lie bracket on $\mathfrak{g}_1\oplus \mathfrak{g}_2$ is defined as the commutator of the left symmetric product \eqref{TwistedProduct}. Namely,
	\begin{eqnarray*}
		&  & [x_1+x_2,y_1+y_2]_\dagger\\	
		& = & (x_1+x_2)\cdot(y_1+y_2)-(y_1+y_2)\cdot(x_1+x_2)\\
		& = & [x_1,y_1]_{\mathfrak{g}_1}+\rho(x_2)(y_1)-\rho(y_2)(x_1)+\theta(x_1)(y_2)-\theta(y_1)(x_2)+[x_2,y_2]_{\mathfrak{g}_2}.
	\end{eqnarray*}
As the commutator of every left symmetric product always allows us to define a structure of Lie algebra, we have that $(\mathfrak{g}_1\oplus \mathfrak{g}_2,[\cdot,\cdot]_\dagger)$ is indeed a Lie algebra. After assuming identities \eqref{Twisted1} and \eqref{Twisted2} we see that the Lie algebra structure on $\mathfrak{g}_1\oplus \mathfrak{g}_2$ may be defined through the representations $(\theta,\rho)$ by setting $[x_1,x_2]_\dagger:=\theta(x_1)(x_2)-\rho(x_2)(x_1)$ for all $x_1\in\mathfrak{g}_1$ and $x_2\in\mathfrak{g}_2$.
	\item[$\iota\iota\iota.$] The symplectic structure on  $\mathfrak{g}_1\oplus \mathfrak{g}_2$ is defined as
	$$\omega(x_1+x_2,y_1+y_2)=\omega_1(x_1,y_1)+\omega_2(x_2,y_2),$$
	for all $x_1,y_1\in\mathfrak{g}_1$ and $x_2,y_2\in\mathfrak{g}_2$. On the one hand, it is clear that $\omega$ is skew-symmetric and non-degenerate. Moreover, it is simple to check that
	\begin{eqnarray*}
		& & \omega([x_1+x_2,y_1+y_2]_\dagger,z_1+z_2)\\	
		& = & \omega_1([x_1,y_1],z_1)+\omega_1(\rho(x_2)(y_1),z_1)-\omega_1(\rho(y_2)(x_1),z_1)\\
		& + & \omega_2([x_2,y_2],z_2)+\omega_2(\theta(x_1)(y_2),z_2)-\omega_2(\theta(y_1)(x_2),z_2).
	\end{eqnarray*}
	As we have that both $\omega_1$ and $\omega_2$ are  scalar 2-cocycles, $\rho(a_2)\in \mathfrak{sp}(\mathfrak{g}_1,\omega_1)$, and $\theta(a_1)\in \mathfrak{sp}(\mathfrak{g}_2,\omega_2)$ for all $a_1\in\mathfrak{g}_1$ and $a_2\in\mathfrak{g}_2$ it follows that
	$$\oint \omega([x_1+x_2,y_1+y_2]_\dagger,z_1+z_2)=0.$$
	That is, $\omega$ is a scalar 2-cocycle on $\mathfrak{g}_1\oplus \mathfrak{g}_2$. On the other hand, the fact that $\rho(a_2),L_{a_1}\in \mathfrak{sp}(\mathfrak{g}_1,\omega_1)$ and $\theta(a_1),L'_{a_2}\in \mathfrak{sp}(\mathfrak{g}_2,\omega_2)$ for all $a_1\in\mathfrak{g}_1$ and $a_2\in\mathfrak{g}_2$ implies that
	$$\omega((x_1+x_2)\cdot(y_1+y_2),z_1+z_2)+\omega(y_1+y_2,(x_1+x_2)\cdot(z_1+z_2))=0.$$ 
	In other words, $\cdot$ is symplectic with respect to $\omega$.

	\item[$\iota\nu.$] Finally, the complex structure $j$ on $\mathfrak{g}_1\oplus \mathfrak{g}_2$ is defined as
	$$j(x_1+x_2)=j_1(x_1)+j_2(x_2),$$
	for all $x_1\in\mathfrak{g}_1$ and $x_2\in\mathfrak{g}_2$. Clearly, $j^2=-1$. We have to check the requirements that we need to ensure that $j$ is an integrable complex structure and it belongs to $Z_{\widetilde{L}}^1(\mathfrak{g}_1\oplus \mathfrak{g}_2,\mathfrak{g}_1\oplus \mathfrak{g}_2)$ where $\widetilde{L}$ is the linear representation induced by the left symmetric product $\cdot$ defined in \eqref{TwistedProduct}. Let us first take a look at the integrability conditions.
	
	On the one side,
	\begin{eqnarray*}
		& & [j(x_1+x_2),j(y_1+y_2)]_\dagger-[x_1+x_2,y_1+y_2]_\dagger \\
		& = & [j_1(x_1),j_1(y_1)]-[x_1,y_1]-\rho(j_2(y_2))(j_1(x_1))+\rho(j_2(x_2))(j_1(y_1))\\
		& - & \rho(x_2)(y_1)+\rho(y_2)(x_1)+[j_2(x_2),j_2(y_2)]-[x_2,y_2]\\
		& + &\theta(j_1(x_1))(j_2(y_2))-\theta(j_1(y_1))(j_2(x_2))+\theta(y_1)(x_2)-\theta(x_1)(y_2).
	\end{eqnarray*}
	On the other side, 
	\begin{eqnarray*}
		& & j([j(x_1+x_2),y_1+y_2]_\dagger+[x_1+x_2,j(y_1+y_2)]_\dagger)\\
		& = & j_1([j_1(x_1),y_1]+[x_1,j_1(y_1)])-j_1(\rho(y_2)(j_1(x_1)))+j_1(\rho(j_2(x_2))(y_1))\\
		& - & j_1(\rho(j_2(y_2))(x_1))+j_1(\rho(x_2)(j_1(y_1)))+j_2([j_2(x_2),y_2]+[x_2,j_2(y_2)])\\
		& + & j_2(\theta(j_1(x_1))(y_2))-j_2(\theta(y_1)(j_2(x_2)))+j_2(\theta(x_1)(j_2(y_2)))-j_2(\theta(j_1(y_1))(x_2)).
	\end{eqnarray*}
	
	Recall that $j_1$ and $j_2$ are integrable complex structures on $\mathfrak{g}_1$ and $\mathfrak{g}_2$, respectively. Therefore, $j$ is an integrable complex structure on $\mathfrak{g}_1\oplus \mathfrak{g}_2$ if and only if
	$$j_1\circ\rho(x_2)\circ j_1+\rho(x_2)=[\rho(j_2(x_2)),j_1],\qquad \textnormal{and}$$
	$$j_2\circ\theta(x_1)\circ j_2+\theta(x_1)=[\theta(j_1(x_1)),j_2]$$
	for all $x_1\in\mathfrak{g}_1$ and $x_2\in\mathfrak{g}_2$. But we must recall that $\rho(x_2)\in \mathfrak{gl}(\mathfrak{g}_1,j_1)$ which means that $\rho(x_2)$ and $j_1$ commute. Thus, $j_1\circ\rho(x_2)\circ j_1+\rho(x_2)=0=[\rho(j_2(x_2)),j_1]$ for all $x_2\in\mathfrak{g}_2$. The same is true for the second identity that we got above because $\theta(x_1)\in \mathfrak{gl}(\mathfrak{g}_2,j_2)$. As consequence, under the assumptions that we have, the complex structure $j$ is always integrable.
	
	Let us now see what happens with the 1-cocycle condition. Note that on the one hand
	\begin{eqnarray*}
		j([x_1+x_2,y_1+x_2]_\dagger) & = & j_1([x_1,y_1])-j_1(\rho(y_2)(x_1))+j_1(\rho(x_2)(y_1))\\
		& + & j_2([x_2,y_2])+j_2(\theta(x_1)(y_2))-j_2(\theta(y_1)(x_2)).
	\end{eqnarray*}
	On the other hand,
	\begin{eqnarray*}
		& & (x_1+x_2)\cdot j(y_1+y_2)-(y_1+y_2)\cdot j(x_1+x_2)\\
		& = & x_1\cdot_1 j_1(y_1)-y_1\cdot_1 j_1(x_1)+\rho(x_2)(j_1(y_1))-\rho(y_2)(j_1(x_1))\\
		& + & x_2\cdot_2 j_2(y_2)-y_2\cdot_2 j_2(x_2)+\theta(x_1)(j_2(y_2))-\theta(y_1)(j_2(x_2)). 
	\end{eqnarray*}
	Recall that $j_1\in Z_{L}^1(\mathfrak{g}_1,\mathfrak{g}_1)$ and $j_2\in Z_{L'}^1(\mathfrak{g}_2,\mathfrak{g}_2)$. Therefore, $j\in Z_{\widetilde{L}}^1(\mathfrak{g}_1\oplus \mathfrak{g}_2,\mathfrak{g}_1\oplus \mathfrak{g}_2)$ if and only if
	$$j_1\circ \rho(x_2)=\rho(x_2)\circ j_1\qquad\textnormal{and}\qquad j_2\circ \theta(x_1)=\theta(x_1)\circ j_2,$$
	for all $x_1\in\mathfrak{g}_1$ and $x_2\in\mathfrak{g}_2$ which always holds true since $\rho(x_2)\in \mathfrak{gl}(\mathfrak{g}_1,j_1)$ and $\theta(x_1)\in \mathfrak{gl}(\mathfrak{g}_1,j_1)$.
	\item[$\nu.$] Let $k_1$ and $k_2$ be the scalar products induced by $(\omega_1,j_1)$ and $(\omega_2,j_2)$, respectively. Then the scalar product $k$ induced by $(\omega,j)$ on $\mathfrak{g}_1\oplus \mathfrak{g}_2$ is given by
	$$k(x_1+x_2,y_1+y_2)=k_1(x_1,y_1)+k_2(x_2,y_2).$$
\end{enumerate}
Summing up:
\begin{theorem}\label{TwistedAlgebra}
	Let $(\mathfrak{g}_1,\omega_1,j_1,\cdot_1)$ and $(\mathfrak{g}_2,\omega_2,j_2,\cdot_2)$ be two special K\"ahler Lie algebras for which there exist two linear representations $\theta:\mathfrak{g}_1\to \mathfrak{kl}(\mathfrak{g}_2,\omega_2,j_2)$ and $\rho:\mathfrak{g}_2\to \mathfrak{kl}(\mathfrak{g}_1,\omega_1,j_1)$ that verify the identities \eqref{Twisted1} and \eqref{Twisted2}. Then the vector space $\mathfrak{g}_1\oplus \mathfrak{g}_2$ equipped with
	\begin{enumerate}
		\item[$\iota.$] the Lie bracket $[\cdot,\cdot]_\dagger$:
		$$[x_1+x_2,y_1+y_2]_\dagger=[x_1,y_1]_{\mathfrak{g}_1}+\rho(x_2)(y_1)-\rho(y_2)(x_1)+\theta(x_1)(y_2)-\theta(y_1)(x_2)+[x_2,y_2]_{\mathfrak{g}_2},$$
		\item[$\iota\iota.$] the non-degenerate scalar 2-cocycle $\omega$:
		$$\omega(x_1+x_2,y_1+y_2)=\omega_1(x_1,y_1)+\omega_2(x_2,y_2),$$
		\item[$\iota\iota\iota.$] the left symmetric product $\cdot$:
		$$(x_1+x_2)\cdot(y_1+y_2)=L_{x_1}(y_1)+\rho(x_2)(y_1)+\theta(x_1)(y_2)+L'_{x_2}(y_2),$$
		\item[$\iota\nu.$] and the integrable complex structure $j$:
		$$j(x_1+x_2)=j_1(x_1)+j_2(x_2),$$
		defines another special K\"ahler Lie algebra.
	\end{enumerate}
\end{theorem}
Motivated for the previous result we set up the following definition:
\begin{definition}
	The special K\"ahler Lie algebra from Theorem \ref{TwistedAlgebra} is called \emph{twisted cartesian product} of $\mathfrak{g}_1$ and $\mathfrak{g}_2$ according to the representations $(\theta,\rho)$.
\end{definition}
Next remarks come in order:
\begin{remark}
	\begin{enumerate}
		\item[$\iota.$] The name twisted cartesian product comes from the fact that if both $\theta$ and $\rho$ are the zero representations, then the special K\"ahler Lie algebra that we get from Theorem \ref{TwistedAlgebra} is the trivial one which is defined through the cartesian product.
		\item[$\iota\iota.$] If $\theta=0$, then the twisted cartesian product becomes in the special K\"ahler Lie algebra obtained as the semi-direct product of $\mathfrak{g}_2$ with $\mathfrak{g}_1$ by means of $\rho$.
		\item[$\iota\iota\iota.$] It is simple to see that both $\mathfrak{g}_1$ and $\mathfrak{g}_2$ are special K\"ahler Lie subalgebras of the twisted cartesian product.
		\item[$\iota\nu.$] If the signature of the scalars product $k_1$ and $k_2$ are $(p_1,q_1)$ and $(p_2,q_2)$, respectively, then the signature of $k$ is $(p_1+p_2,q_1+q_2)$.
	\end{enumerate}
\end{remark}

\begin{corollary}
	Let $(G,\omega,J,\nabla)$ be a special K\"ahler Lie group whose Lie algebra is obtained as the twisted cartesian product of the Lie algebras of two special K\"ahler Lie groups $(G_1,\omega_1,J_1,\nabla_1)$ and $(G_2,\omega_2,J_2,\nabla_2)$ according to the representations $(\theta,\rho)$. Then $\nabla J=0$ if and only if both $\nabla_1 J_1=0$ and $\nabla_2 J_2=0$.
\end{corollary}
\begin{proof}
	Let $\widetilde{L}$ be the linear representation induced by the product $\cdot$ given in \eqref{TwistedProduct}. The result is a straightforward computation that follows from checking what happens when $j=j_1+j_2$ and $\widetilde{L}_{x_1+x_2}$ commute for all $x_1\in\mathfrak{g}_1$ and $x_2\in\mathfrak{g}_2$.
\end{proof}
Let us now introduce a ``double reduction'' process for a special K\"ahler Lie algebra which admits a \emph{complex and non-degenerate} left ideal, that is, a left ideal $I$ of $(\mathfrak{g},\cdot)$ such that $j(I)=I$ and $I$ is symplectic what means that $\omega|_{I\times I}$ is non-degenerate. This motivates the construction that we named twisted cartesian product.
\begin{theorem}\label{TwistedDouble}
	Let $(\mathfrak{g},\omega,j,\cdot)$ be a special K\"ahler Lie algebra which admits a complex and non-degenerate left ideal $I$. Then there exist two special K\"ahler Lie algebras $\mathfrak{g}_1$ and $\mathfrak{g}_2$ together with two Lie algebra representations $\theta:\mathfrak{g}_1\to \mathfrak{kl}(\mathfrak{g}_2,\omega_2,j_2)$ and $\rho:\mathfrak{g}_2\to \mathfrak{kl}(\mathfrak{g}_1,\omega_1,j_1)$  such that the special K\"ahler structure of $\mathfrak{g}$ can be obtained as the twisted cartesian product of $\mathfrak{g}_1$ and $\mathfrak{g}_2$ according to $(\theta,\rho)$.
\end{theorem}
\begin{proof}
	Let us prove that $\mathfrak{g}_1=I$ and $\mathfrak{g}_2=I^{\perp_{\omega}}$ both of them equipped with the special K\"ahler Lie algebra structure of $\mathfrak{g}$ restricted. As $I$ is a left ideal of $(\mathfrak{g},\cdot)$, we have that the following relation holds true
	$$\omega(x\cdot x_2,x_1)=-\omega(x_2,x\cdot x_1)=0,$$
	for all $x\in\mathfrak{g}$, $x_1\in I$, and $x_2\in I^{\perp_{\omega}}$. Thus $x\cdot x_2\in I^{\perp_{\omega}}$ which implies that $I^{\perp_{\omega}}$ is a left ideal of $(\mathfrak{g},\cdot)$ as well. Because the restriction $\omega|_{I\times I}$ is non-degenerate, it is simple to see that $I\cap I^{\perp_{\omega}}=\{0\}$ which means that  $\mathfrak{g}=I\oplus I^{\perp_{\omega}}$. Moreover, the fact that the Lie bracket of $\mathfrak{g}$ is given as the commutator of $\cdot$ implies that both $I$ and $I^{\perp_{\omega}}$ are Lie subalgebras of $\mathfrak{g}$.
	
	Let $L_{x_1}(y_1)=x_1\cdot y_1\in I$  and $L'_{x_2}(y_2)=x_2\cdot y_2$ denote the restriction of the product $\cdot$ to $I$ and $I^{\perp_{\omega}}$, respectively. The properties that we will write below for $I$ with $L:I\to \mathfrak{gl}(I)$ defined by $L_{x_1}(y_1)=x_1\cdot y_1$ for all $x_1, y_1\in I$ are also true for $I^{\perp_{\omega}}$ with $L':I^{\perp_{\omega}}\to \mathfrak{gl}(I^{\perp_{\omega}})$ defined by $L'_{x_2}(y_2)=x_2\cdot y_2$ for all $x_2, y_2\in I^{\perp_{\omega}}$.
	\begin{enumerate}
		\item[$\iota.$] Because $\cdot$ is a left symmetric product compatible with the Lie algebra structure of $\mathfrak{g}$ and $I$ is a left ideal, the map $L:I\to \mathfrak{gl}(I)$ is a Lie algebra representation. Moreover, the facts that $\cdot$ is symplectic with respect to $\omega$ and $j\in Z_L^1(\mathfrak{g},\mathfrak{g})$ respectively implies that
		$$\omega_1(L_{x_1}(y_1),z_1)+\omega_1(y_1,L_{x_1}(z_1))=0\qquad\textnormal{and}\qquad j_1\in Z_L^1(I,I),$$
		for all $x_1, y_1, z_1\in I$. Here $\omega_1=\omega|_{I\times I}$ and $j_1=j|_{I}$. For the case of $I^{\perp_{\omega}}$ we set $\omega_2=\omega|_{I^{\perp_{\omega}}\times I^{\perp_{\omega}}}$ and $j_2=j|_{I^{\perp_{\omega}}}$.
		\item[$\iota\iota.$] Let us now define $\theta:I\to \mathfrak{gl}(I^{\perp_{\omega}})$ and $\rho:I^{\perp_{\omega}}\to\mathfrak{gl}(I)$ as $\theta(x_1)(x_2)=x_1\cdot x_2$ and $\rho(x_2)(x_1)=x_2\cdot x_1$ for all $x_1\in I$ and $x_2\in I^{\perp_{\omega}}$, respectively. On the one hand, as $\cdot$ is a left symmetric product we have
		\begin{eqnarray*}
			(x_1\cdot y_1-y_1\cdot x_1)\cdot x_2 & = & [x_1,y_1]\cdot x_2\\
			& = & x_1 \cdot(y_1\cdot x_2)-y_1\cdot(x_1\cdot x_2)\\
			& = & \theta(x_1)(\theta(y_1)(x_2))-\theta(y_1)(\theta(x_1)(x_2)).
		\end{eqnarray*}
		On the other hand, the fact that $\cdot$ is symplectic with respect to $\omega$ implies
		$$0=\omega_2(x_1\cdot x_2,y_2)+\omega_2( x_2,x_1\cdot y_2)=\omega_2(\theta(x_1)(x_2),y_2)+\omega_2(x_2,\theta(x_1)(y_2)).$$
		In other words, $\theta:I\to \mathfrak{sp}(I^{\perp_{\omega}},\omega_2)$ is a Lie algebra representation. From the equalities
		$$[x_2,y_2]\cdot x_1= x_2 \cdot(y_2\cdot x_1)-y_2\cdot(x_2\cdot x_1)\quad\textnormal{and}\quad \omega_1(x_2\cdot x_1,y_1)+\omega_1(x_1,x_2\cdot y_1)=0,$$
		it follows that $\rho:I^{\perp_{\omega}}\to \mathfrak{sp}(I,\omega_1)$ is a Lie algebra representation as well.
		\item[$\iota\iota\iota.$] The identities
		$$x_1\cdot(x_2\cdot y_1)-x_2\cdot(x_1\cdot y_1)=(x_1\cdot x_2)\cdot y_1-(x_2\cdot x_1)\cdot y_1,\qquad\textnormal{and}$$
		$$x_1\cdot(x_2\cdot y_2)-x_2\cdot(x_1\cdot y_2)=(x_1\cdot x_2)\cdot y_2-(x_2\cdot x_1)\cdot y_2,$$
		must be satisfied for all $x_1,y_1\in I$ and $x_2,y_2\in I^{\perp_{\omega}}$. This is equivalent to require that the representations $\theta:I\to \mathfrak{sp}(I^{\perp_{\omega}},\omega_2)$ and $\rho:I^{\perp_{\omega}}\to \mathfrak{sp}(I,\omega_1)$ verify the equations
		\begin{equation*}
		L_{x_1}\circ\rho(x_2)-\rho(x_2)\circ L_{x_1}=\rho(\theta(x_1)(x_2))-L_{\rho(x_2)(x_1)},\qquad\textnormal{and}
		\end{equation*}
		\begin{equation*}
		L'_{x_2}\circ\theta(x_1)-\theta(x_1)\circ L'_{x_2}=\theta(\rho(x_2)(x_1))-L'_{\theta(x_1)(x_2)}.
		\end{equation*}
		\item[$\iota\nu.$] Finally, as $j\in Z_L^1(\mathfrak{g},\mathfrak{g})$ we get that
		\begin{eqnarray*}
			\theta(x_1)(j_2(x_2))-\rho(x_2)(j_1(x_1)) & = & x_1 \cdot j_2(x_2)-x_2\cdot j_1(x_1)\\
			& = & x_1 \cdot j(x_2)-x_2\cdot j(x_1)\\
			& = & j([x_1,x_2])\\
			& = & j(x_1 \cdot x_2)-j(x_2\cdot x_1)\\
			& = & j_2(\theta(x_1)(x_2))-j_1(\rho(x_2)(x_1)).
		\end{eqnarray*}
		It happens if and only if $\theta(x_1)\circ j_2=j_2\circ \theta(x_1)$ and $\rho(x_2)\circ j_1=j_1\circ \rho(x_2)$ for all $x_1\in I$ and $x_2\in I^{\perp_{\omega}}$. Therefore, $\theta:I\to \mathfrak{kl}(I^{\perp_{\omega}},\omega_2,j_2)$ and $\rho:I^{\perp_{\omega}}\to \mathfrak{kl}(I,\omega_1,j_1)$. 
	\end{enumerate}
	Clearly $\omega_1$ and $\omega_2$ are non-degenerate scalar 2-cocycles, and, $j_1$ and $j_2$ are integrable complex structures. Hence, both $(I,\omega_1,j_1,\cdot)$ and $(I^{\perp_{\omega}},\omega_2,j_2,\cdot')$ are special K\"ahler Lie algebras and  $(\mathfrak{g},\omega,j,\cdot)$ is actually the twisted cartesian product of $I$ and $I^{\perp_{\omega}}$ according to the Lie algebra representations $(\theta,\rho)$ defined above.
\end{proof}

\begin{corollary}
If in the hypothesis of Theorem \ref{TwistedDouble} we assume that $I$ is a bilateral ideal of $(\mathfrak{g},\cdot)$, then $\theta=0$ which means that $\mathfrak{g}$ is obtained as the semi-direct product of $\mathfrak{g}_2$ with $\mathfrak{g}_1$ by means $\rho$.
\end{corollary}
\begin{proof} As $I$ is a bilateral ideal we get that $\theta(x_1)(x_2)=x_1\cdot x_2\in I$ for all $x_1\in I$ and $x_2\in I^{\perp_{\omega}}$ and hence $\omega_2(\theta(x_1)(x_2),y_2)=0$ for all $y_2\in I^{\perp_{\omega}}$. Therefore, as $\omega_2$ is non-degenerate we deduce that $\theta(x_1)(x_2)=0$ for all $x_1\in I$ and $x_2\in I^{\perp_{\omega}}$. That is, $\theta=0$.
\end{proof}

\begin{example}\label{GoodExample}
	Consider the special K\"ahler Lie algebra $\mathfrak{g}_3$ in dimension $4$ associated to the special K\"ahler Lie group $G_3$ given in Example \ref{KeyExample1}. Here we set $\mathfrak{g}_3\cong \textnormal{Vect}_\mathbb{R}\lbrace e_1,e_2,e_3,e_4\rbrace$ with nonzero Lie brackets
	$$[e_1,e_2]=[e_1,e_4]=[e_3,e_2]=[e_3,e_4]=e_2-e_4.$$
	On this Lie algebra we have the structure of special K\"ahler Lie algebra:
	\begin{enumerate}
		\item[$\iota.$] symplectic form $\omega=e_1^\ast \wedge e_2^\ast - e_3^\ast\wedge e_4^\ast$,
		\item[$\iota\iota.$] integrable complex structure $j(e_1)=e_2$ and $j(e_3)=e_4$, and
		\item[$\iota\iota\iota.$] left symmetric product
		\begin{center}
			\begin{tabular}{c|c|c|c|c}
				$\overline{\cdot}$	& $e_1$ & $e_2$ & $e_3$ & $e_4$ \\
				\hline
				$e_1$	& $-e_1+e_3$ & $e_2-e_4$ & $-e_1+e_3$ & $e_2-e_4$ \\
				\hline
				$e_2$	& $0$ & $2e_1-2e_3$  & $0$ & $2e_1-2e_3$ \\
				\hline
				$e_3$	& $-e_1+e_3$ & $e_2-e_4$ & $-e_1+e_3$ & $e_2-e_4$ \\
				\hline
				$e_4$	& $0$ & $2e_1-2e_3$  & $0$ & $2e_1-2e_3$ \\
			\end{tabular}
		\end{center}
	\end{enumerate}
	Let us now consider the model space of special K\"ahler Lie group $((\mathbb{R}^{2n},+),\omega_0,J_0,\nabla^0)$. Note that the special K\"ahler Lie algebra associated to this Lie group is $(\mathbb{R}^{2n},\omega_0,J_0,\cdot^0)$ where $x\cdot^0 y=0$ since $\nabla^0_{\partial_i}{\partial_j}=0$. 
	
	A straightforward computation allows us to get that an element $D\in\mathfrak{sp}(\mathfrak{g}_3,\omega)$ such that $[D,j]=0$ has the form:
	$$D= \left( 
	\begin{array}{cccc}
	0 & a & b & c\\
	-a & 0 & -c & b\\
	b & -c & 0 & d\\
	c & b & -d & 0
	\end{array}%
	\right),\qquad a,b,c,d\in\mathbb{R}.$$
	
	On the one hand, let $T:\mathbb{R}^{2n}\to \mathbb{R}$ be a linear transformation and define the linear Lie algebra representation $\rho: \mathbb{R}^{2n}\to \mathfrak{kl}(\mathfrak{g}_3,\omega,j)$ as $\rho(x)=T(x)D$ for $D$ as fixed above. On the other hand, let $\theta:\mathfrak{g}_3\to \mathfrak{kl}(\mathbb{R}^{2n},\omega_0,J_0)$ be the trivial representation; $\theta=0$. To have that the product \eqref{TwistedProduct} is a left symmetric product on $\mathfrak{g}=\mathfrak{g}_3\oplus \mathbb{R}^{2n}$ we just  need to find the required conditions for which the equation \eqref{Twisted1} holds true. Namely,
	\begin{equation}\label{ExampleTwisted}
	T(x_2)(L_{x_1}\circ D-D\circ L_{x_1})=-T(x_2)L_{D(x_1)}.
	\end{equation}
	Here $L$ is the linear representation induced by the left symmetric product $\overline{\cdot}$. If $x_2\in \textnormal{ker}(T)$, then the identity \eqref{ExampleTwisted} trivially holds. However, if $x_2\notin \textnormal{ker}(T)$ and $x_1=e_1$, then with a straightforward computation we get that \eqref{ExampleTwisted} is true if and only if
	$$D= \left( 
	\begin{array}{cccc}
	0 & a & 0 & a\\
	-a & 0 & -a & 0\\
	0 & -a & 0 & -a\\
	a & 0 & a & 0
	\end{array}%
	\right),\qquad a\in\mathbb{R}.$$ 
	If $D$ takes this form, then it is simple to check that the identity \eqref{ExampleTwisted} always holds for $x_1\in\{e_2,e_3,e_4\}$. Therefore, the vector space $\mathfrak{g}=\mathfrak{g}_3\oplus \mathbb{R}^{2n}$ has a non-trivial structure of special K\"ahler Lie algebra given by
	\begin{enumerate}
		\item[$\iota.$] the nonzero Lie brackets:
		$$[e_1,e_2]=[e_1,e_4]=[e_3,e_2]=[e_3,e_4]=e_2-e_4\qquad\textnormal{and}\qquad [e_j,\hat{e}_k]=-T(\hat{e}_k)D(e_j)$$
		for all $j=1,\cdots,4$ and $k=1,2,\cdots, 2n$
		\item[$\iota\iota.$] the symplectic form $\displaystyle \widetilde{\omega}=e_1^\ast \wedge e_2^\ast - e_3^\ast\wedge e_4^\ast+\sum_{k=1}^{n}\hat{e}_k^\ast\wedge \hat{e}_{n+k}^\ast$,
		\item[$\iota\iota\iota.$] the integrable complex structure $\widetilde{j}(e_1)=e_2$, $\widetilde{j}(e_3)=e_4$, and $\widetilde{j}(\hat{e}_k)=\hat{e}_{n+k}$ for all $k=1,2,\cdots, n$; and
		\item[$\iota\nu.$] the left symmetric product
		\begin{center}
			\begin{tabular}{c|c|c|c|c|c}
				$\cdot$	&  $e_1$ & $e_2$ & $e_3$ & $e_4$ & $\hat{e}_j$ \\
				\hline
				$e_1$	&  $-e_1+e_3$ & $e_2-e_4$ & $-e_1+e_3$ & $e_2-e_4$ & $0$ \\
				\hline
				$e_2$	&  $0$ & $2e_1-2e_3$  & $0$ & $2e_1-2e_3$ & $0$ \\
				\hline
				$e_3$	&  $-e_1+e_3$ & $e_2-e_4$ & $-e_1+e_3$ & $e_2-e_4$ & $0$\\
				\hline
				$e_4$	&  $0$ & $2e_1-2e_3$  & $0$ & $2e_1-2e_3$ & $0$ \\
				\hline
				$\hat{e}_k$	& $T(\hat{e}_k)D(e_1)$ & $T(\hat{e}_k)D(e_2)$  & $T(\hat{e}_k)D(e_3)$ & $T(\hat{e}_k)D(e_4)$ & $0$ \\
			\end{tabular}
		\end{center}
		for all $j,k=1,2,\cdots, 2n$.
	\end{enumerate}
Since $\widetilde{L}_{e_1}\circ \widetilde{j}\neq \widetilde{j}\circ \widetilde{L}_{e_1}$ we have that the left invariant flat affine symplectic connection $\widetilde{\nabla}$ determined by $\cdot$ satisfies $\widetilde{\nabla}\widetilde{J}\neq 0$. Finally, the signature of the scalar product on $\mathfrak{g}$ induced by $(\widetilde{\omega},\widetilde{j})$ is $(2,2n+2)$.
\end{example}

\section{A double extension of a special K\"ahler Lie algebra}
In \cite{V} was introduced a method for constructing flat affine symplectic Lie algebras of dimension $2n+2$ starting from a flat affine symplectic Lie algebra of dimension $2n$. This method uses the cohomology for left symmetric algebras developed in \cite{N} and it is called a \emph{double extension of a flat affine symplectic Lie algebra}. In this section we will use such a construction for inducing an easy way of getting special K\"ahler Lie algebras.

For the purposes of this section let us assume that $(\mathfrak{g},\omega,j,\cdot)$ is a special K\"ahler Lie algebra such that the scalar product $k$ induced by $(\omega,j)$ is positive definite. In the same spirit of \cite{DM} and \cite{V} we may  find a way of obtaining a reduction process of a special K\"ahler Lie algebra structure.
\begin{lemma}[Reduction]\label{reduction}
Suppose that $(\mathfrak{g},\omega,j,\cdot)$ is a special K\"ahler Lie algebra. Let $I$ be totally isotropic bilateral ideal of $(\mathfrak{g},\omega,\cdot)$. Then
\begin{enumerate}
	\item[$\iota.$] The product $\cdot$ in $I$ is null, $I\cdot  I^{\perp_\omega}=0$, and $ I^{\perp_\omega}$ is a left ideal.
	\item[$\iota\iota.$] $ I^{\perp_\omega}$ is a right ideal if and only if, $ I^{\perp_\omega}\cdot I=0$.
	\item[$\iota\iota\iota.$] If $ I^{\perp_\omega}$ is a bilateral ideal of $(\mathfrak{g},\cdot)$, then the canonical sequences
	\begin{equation}\label{secu1}
	0\longrightarrow I\hookrightarrow I^{\perp_\omega}\longrightarrow I^{\perp_\omega}/ I=B\longrightarrow 0,
	\end{equation}
	\begin{equation}\label{secu3}
	0\longrightarrow I\hookrightarrow \mathfrak{g}\longrightarrow \mathfrak{g}/ I\longrightarrow 0,
	\end{equation}
	are sequences of left symmetric algebras. Moreover, the quotient Lie algebra $B=I^{\perp_\omega}/I$ admits a canonical structure of special K\"ahler Lie algebra.
\end{enumerate}
\end{lemma}
\begin{proof}
Items $\iota.$ and $\iota\iota.$ are straightforward computations; see for instance \cite{V}. Let us verify $\iota\iota\iota.$ with more details. Suppose that $I^{\perp_\omega}$ is a bilateral ideal of $(\mathfrak{g},\cdot)$. Given that $I\subset I^{\perp_\omega}$ the fact that $\omega$ is a non-degenerate scalar 2-cocycle implies that $I$ is Abelian. Thus, the quotient vector space $B=I^{\perp_\omega}/I$ inherits a natural structure of symplectic Lie algebra passing to the quotient; see \cite{DM}. More precisely, $B$ has a structure of Lie algebra given by
$$[x+I,y+I]=[x,y]+I=(x\cdot y-y\cdot x)+I,\qquad x,y\in I^{\perp_\omega}.$$
As $\omega$ is a nondegenerate scalar 2-cocycle, it induces on $I^{\perp_\omega}$ a bilinear form of radical $I$. Hence, we have that $\left.\omega\right\vert_{I^{\perp_\omega}\times I^{\perp_\omega}}$ defines, passing to the quotient, a nondegenerate and skew-symmetric bilinear form $\omega'$ on $B=I^{\perp_\omega}/I$ which is also a scalar 2-cocycle of the Lie algebra $B$. This symplectic form is given by $\omega'(x+I,y+I)=\omega(x,y)$ for all $x,y\in I^{\perp_\omega}$.

If we denote the class of $x\in I^{\perp_\omega}$ module $I$ by $\overline{x}=x+I$, then the left symmetric product
$$\overline{x}\cdot\overline{y}=(x+I)\cdot (y+I)=x\cdot y+I=\overline{x\cdot y},$$
satisfies
$$\omega'(\overline{x}\cdot\overline{y},\overline{z})+\omega'(\overline{y},\overline{x}\cdot\overline{z})=\omega'(\overline{x\cdot y},\overline{z})+\omega'(\overline{y},\overline{x\cdot z})=\omega(x\cdot y,z)+\omega(y,x\cdot z)=0,$$
for all $x,y,z\in I^{\perp_\omega}$. 

It is worth noticing that as the scalar product $k$ on $\mathfrak{g}$ induced by $(\omega,j)$ is positive definite then for each $a\in I\backslash\lbrace 0\rbrace$ we get that $k(a,a)=\omega(a,j(a))>0$ so that $j(a)\notin I^{\perp_\omega}$. In consequence, it follows that $I^{\perp_\omega}\cap j(I)=\lbrace 0\rbrace$ and more importantly to us we obtain that $(I^{\perp_\omega}\cap j(I)^{\perp_\omega})\oplus I=I^{\perp_\omega}$ where $I^{\perp_\omega}\cap j(I)^{\perp_\omega}$ is a subspace in $I^{\perp_\omega}$ invariant by $j$; visit \cite{DM} for further details. This implies that $j$ is well restricted to $I^{\perp_\omega}\cap j(I)^{\perp_\omega}$ which in turn may be identified with $B$. After assuming this identification we find a unique way of viewing each class $\overline{x}=x+I$ with $x\in I^{\perp_\omega}\cap j(I)^{\perp_\omega}$, thus obtaining that the map $j':B\to B$ given as $j'(\overline{x})=j(x)+I=\overline{j(x)}$ defines an integrable complex structure on $B$ such that $(B,\omega',j')$ is a K\"ahler Lie algebra. Moreover, we get that
$$j'(\overline{[x,y]})=\overline{j([x,y])}=\overline{x\cdot j(y)-y\cdot j(x)}=\overline{x\cdot j(y)}-\overline{y\cdot j(x)}=\overline{x}\cdot \overline{j(y)}-\overline{y} \cdot \overline{j(x)}.$$
In conclusion, the triple $(B,\omega',j',\overline{\cdot})$ is a special K\"ahler Lie algebra.

\end{proof}

\begin{Assumption}
Assume that both $I$ and $I^{\perp_\omega}$ are bilateral ideals of $(\mathfrak{g},\cdot)$ with $\dim I=1$. 
\end{Assumption}
It is clear that if $\dim I=1$, then we are in the conditions of Lemma \ref{reduction}. Let $B=I^{\perp_\omega}/I$ denote the special K\"ahler Lie algebra obtained from the reduction. Suppose that $I=\mathbb{R}e$ with $k(e,e)=1$ and set $d=j(e)$ so that $\mathbb{R}d$ is a 1-dimensional subspace in $\mathfrak{g}$ such that $\omega(e,d)=1$. As $B\approx I^{\perp_\omega}\cap j(I)^{\perp_\omega}$ is invariant by $j$ and $j(I)=\mathbb{R}d$ then we can identify $\mathfrak{g}\approx \mathbb{R}e\oplus B\oplus \mathbb{R}d$. From now on, the left symmetric products of $B$ and $\mathfrak{g}$ will be denoted by $\cdot$ and $\diamond$, respectively. Also, if $A\in \mathfrak{gl}(B)$ then we denote by $A^\ast$ the adjoint map with respect to $\omega'$ associated to $A$, that is, $A^\ast\in \mathfrak{gl}(B)$ verifies $\omega'(A(x),y)=\omega'(x,A^\ast(y))$ for all $x,y\in B$. Under the previous identifications we have the following facts. See \cite{V} for more details about $\iota.$, $\iota\iota.$, and $\iota\iota\iota.$ stated below.
\begin{enumerate}
	\item[$\iota.$] Lie algebra structure of $\mathfrak{g}$:
	\begin{eqnarray}\label{bracketdoble2}
	&  & [d,e]=\mu e\nonumber\\
	&  & [d,x]=\omega'(z_0,x)e+D(x)\\
	&  & [x,y]=\omega'((u+u^*)(x),y)e+[x,y]_B\nonumber,\qquad x,y\in B.
	\end{eqnarray}
	\item[$\iota\iota.$] Non-degenerate scalar 2-cocycle $\omega$:
	$$\omega|_B=\omega'\qquad \omega(e,d)=1\qquad\textnormal{and}\qquad \textnormal{Vect}_\mathbb{R}\lbrace e,d\rbrace \perp_{\omega} B.$$
	\item[$\iota\iota\iota.$] Left symmetric product $\diamond$ compatible with the Lie algebra structure of $\mathfrak{g}$ and symplectic with respect to $\omega$:
	\begin{eqnarray}\label{productdoble2}
	&  & e\diamond x=x\diamond e=e\diamond e=0\nonumber\\
	&  & x\diamond y=\omega'(u(x),y)e+x\cdot y\nonumber\\
	&  & d\diamond x=\omega'(x_0,x)e+(D+u)(x)\nonumber\\
	&  & x\diamond d=\omega'(x_0-z_0,x)e+u(x)\\
	&  & d\diamond e=\lambda e\nonumber\\
	&  & e\diamond d=(\lambda-\mu)e\nonumber\\
	&  & d\diamond d=\beta e+x_0-\lambda d\nonumber,
	\end{eqnarray}
\end{enumerate} 
where $\lambda,\mu,\beta\in\mathbb{R}$, $x_0,z_0\in B$, $D\in\mathfrak{gl}(B)$, and $u\in Z_{L}^1(B,B)$ such that $D+u\in\mathfrak{sp}(B,\omega')$. All these parameters must verify the following algebraic conditions:
\begin{enumerate}
	\item $\lambda=\mu$ or $\lambda=\dfrac{\mu}{2}$.
	\item $[u,D]_{\mathfrak{gl}(B)}=u^2+\lambda u-R_{x_0}$.
	\item $D^\ast(x_0-z_0)-2u^\ast(x_0)-2\lambda(x_0-z_0)+(\lambda-\mu)z_0=0$.
	\item $D(x)\cdot y+x\cdot D(y)-D(x\cdot y)=u(x\cdot y)-x\cdot u(y)$.
	\item $(\lambda-\mu)(u+u^\ast)(x)-2(u\circ u^\ast)(x)=(L_x+R_x^\ast)(x_0-z_0)$.
\end{enumerate}
Here $R_x:B\to B$ is defined as $R_x(y)=L_y(x)=y\cdot x$ for all $x,y\in B$. For more details see \cite{V}.

Let us now see what happens with the integrable complex structure $j$. Note that because of the identifications we have done above we know that $j$ satisfies
\begin{equation}\label{ComplexStructureExtended}
j|_B=j'\quad\textnormal{and}\quad j(e)=d,
\end{equation}
so that $j(d)=-e$. To look at the integrability of $j$ with respect to the Lie bracket \eqref{bracketdoble2} we need to analyze all possible cases:
\begin{itemize}
\item the equality $[j(e),j(d)]-[e,d]=j[j(e),d]+j[e,j(d)]$ always holds.
\item The identity $[j(e),j(x)]-[e,x]=j[j(e),x]+j[e,j(x)]$ holds if and only if
$$\omega'(z_0,x)=\omega'(z_0,j'(x))=0\quad\textnormal{and}\qquad [D,j'](x)=0,\qquad x\in B.$$
That is, $[D,j']=0$ and because $\omega'$ is non-degenerate, we also have $z_0=0$.
\item Analogously, the equality $[j(d),j(x)]-[d,y]=j[j(d),y]+j[d,j(x)]$ is satisfied if and only if $z_0=0$ and $[D,j']=0$.
\item Finally, given that $j'$ is integrable, the identity $[j(x),j(y)]-[x,y]=j[j(x),y]+j[x,j(y)]$ is true if and only if 
$$\omega'((u+u^\ast)(j'(x)),j'(y))=\omega'((u+u^*)(x),y)\qquad\textnormal{and}$$
$$\omega'((u+u^\ast)(j'(x)),y)=-\omega'((u+u^*)(x),j'(y)),$$
for all $x,y\in B$. As we have that $j'\in \mathfrak{sp}(B,\omega')$ and $\omega'$ is non-degenerate, the previous conditions are verified if and only if $[u+u^\ast,j']=0$.
\end{itemize}
Let us now look at the condition $j\in Z_{L}^1(\mathfrak{g},\mathfrak{g})$, where $L$ denotes the linear representation determined by the left symmetric product $\diamond$ given in \eqref{productdoble2}.
\begin{itemize}
\item The equality $j([e,d])=e\diamond j(d)-d\diamond j(e)$ holds if and only if $-\mu d=-(\beta e+x_0-\lambda d)$. This implies that $\beta=0$, $x_0=0$, and $\lambda=-\mu$. Note that a strong condition in the double extension of a flat affine symplectic Lie algebra is $\lambda=\mu$ or $\lambda=\dfrac{\mu}{2}$. Thus $\lambda=\mu=0$ as well.
\item The identity $j([e,x])=e\diamond j(x)-x\diamond j(e)$ is satisfied if and only if $x\diamond d=0$ for all $x\in B$. As $z_0=x_0=0$ we get that $u=0$. 
\item Given that $z_0=x_0=0$ and $u=0$, the equality $j([d,x])=d\diamond j(x)-x\diamond j(d)$ is true if and only if $(j'\circ D)(x)=(D\circ j')(x)$ for all $x\in B$. So, $[D,j']=0$.
\item Note that as consequence of the algebraic condition (4) presented above it follows that $D$ must be a derivation with respect to the left symmetric product since $u=0$.
\item Finally, because $j'\in Z_L^1(B,B)$, the identity $j([x,y])=x\diamond j(y)-y\diamond j(y)$ always holds.
\end{itemize}
Summing up, we have the following results:
\begin{proposition}\label{PositiveResult}
	Let $(\mathfrak{g},\omega,j,\cdot)$ be a special K\"ahler Lie algebra. Assume that $I=\mathbb{R}e$ with $k(e,e)=1$ is a $1$-dimensional subspace in $\mathfrak{g}$ such that both $I$ and $I^{\perp_\omega}$ are bilateral ideals of $(\mathfrak{g},\cdot)$. If $B=I^{\perp_\omega}/I$ denotes the special K\"ahler Lie algebra obtained from the reduction through $I$, then when setting $d=j(e)$ the left symmetric product of $\mathfrak{g}$ is given by
	$$\diamond|_B=\cdot \qquad e\diamond \mathfrak{g}=\mathfrak{g}\diamond e=\mathfrak{g}\diamond d=d\diamond(\mathbb{R}e\oplus \mathbb{R}d) =\lbrace 0\rbrace\qquad\textnormal{and}\qquad d\diamond x=D(x),\qquad x\in B,$$ 
where $D\in \mathfrak{sp}(B,\omega')$ is a derivation of the left symmetric algebra $(B,\cdot)$ verifying $[D,j']=0$.
\end{proposition}
Reciprocally, we get a method for constructing special K\"ahler Lie algebras:
\begin{theorem}\label{doubleKahler}
Let $(\mathfrak{g},\omega,j,\cdot)$ be a special K\"ahler Lie algebra and let $D\in \mathfrak{sp}(\mathfrak{g},\omega)$ be a derivation of the left symmetric algebra $(\mathfrak{g},\cdot)$ verifying $[D,j]=0$. Then the vector space 
$\hat{\mathfrak{g}}:=\mathbb{R}e\oplus \mathfrak{g} \oplus \mathbb{R}d$ equipped with
\begin{enumerate}
\item[$\iota.$] the Lie bracket $[\cdot,\cdot]$:
$$[\cdot,\cdot]|_\mathfrak{g}=[\cdot,\cdot]_\mathfrak{g},\qquad e\in \mathfrak{z}(\hat{\mathfrak{g}})\qquad \textnormal{and}\qquad [d,x]=D(x),\qquad x\in\mathfrak{g}$$
\item[$\iota\iota.$] the non-degenerate scalar 2-cocycle $\widetilde{\omega}$:
$$\widetilde{\omega}|_\mathfrak{g}=\omega,\qquad \widetilde{\omega}(e,d)=1\qquad\textnormal{and}\qquad \textnormal{Vect}_\mathbb{R}\lbrace e,d\rbrace \perp_{\widetilde{\omega}} \mathfrak{g}$$
\item[$\iota\iota\iota.$] the left symmetric product $\diamond$:
$$\diamond|_\mathfrak{g}=\cdot \qquad e\diamond \hat{\mathfrak{g}}=\hat{\mathfrak{g}}\diamond e=\hat{\mathfrak{g}}\diamond d=d\diamond(\mathbb{R}e\oplus \mathbb{R}d) =\lbrace 0\rbrace\qquad\textnormal{and}\qquad d\diamond x=D(x),\qquad x\in\mathfrak{g}$$ 
\item[$\iota\nu.$] and the integrable complex structure $\widetilde{j}$:
$$\widetilde{j}|_\mathfrak{g}=j\qquad\textnormal{and}\qquad \widetilde{j}(e)=d
,$$
\end{enumerate}
defines another special K\"ahler Lie algebra.
\end{theorem}
\begin{definition}
The Lie algebra $(\hat{\mathfrak{g}},\widetilde{\omega},\widetilde{j},\diamond)$ obtained in Theorem \ref{doubleKahler} is called the \emph{double extension} of the special K\"ahler Lie algebra $(\mathfrak{g},\omega,j,\cdot)$ according to $D$.
\end{definition}
\begin{remark}
 \begin{enumerate}
 \item[$\iota.$] If $k$ denotes the scalar product on $\mathfrak{g}$ induced by $(\omega,j)$, then the scalar product $\widetilde{k}$ on the double extension $\hat{\mathfrak{g}}$ which is induced by $(\widetilde{\omega},\widetilde{j})$ can be seen as
 $$\widetilde{k}=\left( 
 \begin{array}{ccc}k &  &  \\
 & 1 &\\
 & & 1
 \end{array}%
 \right).$$
 \item[$\iota\iota.$] The requirement of positive definiteness of $k$ it is completely necessary to prove the reduction procedure from Lemma \ref{reduction}. If we let the scalar product $k$ to have signature $(p,q)$ with both $p,q$ nonzero, then the reduction procedure and Proposition \ref{PositiveResult} are not true in general. However, if $(\mathfrak{g},\omega,j,\cdot)$ is a special K\"ahler Lie algebra where the scalar product $k$ induced by $(\omega,j)$ is not necessarily positive definite, then the double extension process stated in Theorem \ref{doubleKahler} is still true. What we need to do for proving this claim is to use the double extension process of a flat affine symplectic Lie algebra introduced in \cite{V} and extend the complex structure as in equation \eqref{ComplexStructureExtended}. The requirements of integrability and cohomology property of the complex structure extended are exactly the same that we got before. In this case, if the signature of $k$ is $(p,q)$, then the signature of $\widetilde{k}$ is $(p+2,q)$.
 \end{enumerate}	
\end{remark}

\begin{corollary}
	If $(G,\omega,J,\nabla)$ is a simply connected special K\"ahler Lie group whose Lie algebra is obtained as a double extension, then $G$ is identified with a Lie subgroup of $\mathfrak{g}\rtimes_{\textnormal{Id}}\textnormal{Sp}(\mathfrak{g},\omega_e)$ containing a nontrivial $1$-parameter subgroup formed by  central translations. In particular, if $\nabla J=0$, then such a subgroup is contained in $\mathfrak{g}\rtimes_{\textnormal{Id}}\textnormal{KL}(\mathfrak{g},\omega_e,J_e)$.
\end{corollary}
\begin{proof}
Let $(\mathfrak{g},\omega,j,\cdot)$ be the special K\"ahler Lie algebra associated to $(G,\omega,J,\nabla)$. If $\mathfrak{g}$ is obtained as a double extension of a special pseudo-K\"ahler Lie algebra $B$ according to $D$, then it decomposes as $\mathfrak{g}=\mathbb{R}e\oplus B\oplus \mathbb{R}d$ where $(\omega,j,\cdot)$ are given like $(\widetilde{\omega},\widetilde{j},\diamond)$ in Theorem \ref{doubleKahler}. As $e\diamond \mathfrak{g}=\lbrace 0\rbrace$, it is clear that $\textnormal{Ker}(L)\neq \lbrace 0\rbrace$ since $L_e=0$. Given that $G$ is simply connected, there exists a Lie group homomorphism $\rho: G\to \mathfrak{g}\rtimes_{\textnormal{id}}\textnormal{Sp}(\mathfrak{g},\omega)$ which is determined by the expression
$$\rho(\exp_G(x))=\left(\sum_{m=1}^\infty \dfrac{1}{m!}(L_x)^{m-1}(x),\sum_{m=0}^\infty \dfrac{1}{m!}(L_x)^m\right).$$
See \cite{V} for more details about such a Lie group homomorphism. Therefore, as $L_e=0$, we have that $\rho$ determines a nontrivial $1$-parameter subgroup $H$ of $\rho(G)\approx G$ formed by central translations which is induced by $t\mapsto \textsf{exp}_G(te)$ and given by
$$H=\left\lbrace \rho(\textsf{exp}_G(te))=(te,\textnormal{Id}_\mathfrak{g}):\ t\in\mathbb{R}\right\rbrace.$$
In particular, if $\nabla J=0$ then it follows from Theorem \ref{etale} that $H$ is actually contained in $\mathfrak{g}\rtimes_{\textnormal{Id}}\textnormal{KL}(\mathfrak{g},\omega,j)$.
\end{proof}
\begin{corollary}\label{NablaJ}
Let $(G_1,\omega_1,J_1,\nabla_1)$ and $(G_2,\omega_2,J_2,\nabla_2)$ be two special K\"ahler Lie groups such that the Lie algebra of $G_1$ is obtained as a double extension from the Lie algebra of $G_2$ according with $D$. Then $\nabla J_1=0$ if and only if $\nabla J_2=0$.
\end{corollary}
\begin{proof}
The result follows from the condition $[j_2,D]=0$.
\end{proof}

It is well known that a left invariant flat affine symplectic connection on a connected Lie group is geodesically complete if and only if the group is unimodular (see for instance \cite{Ba,V}). So, the following result is clear:
\begin{corollary}
Let $(G_1,\omega_1,J_1,\nabla_1)$ and $(G_2,\omega_2,J_2,\nabla_2)$ be two special K\"ahler Lie groups such that the Lie algebra of $G_1$ is obtained as a double extension from the Lie algebra of $G_2$ according with $D$. Then $\nabla_1$ is geodesically complete if and only if $\nabla_2$ is geodesically complete and $\textnormal{tr}(D)=0$.
\end{corollary}

\begin{corollary}
Let $(G,\omega,J,\nabla)$ be a special K\"ahler Lie group such that $\nabla$ is bi-invariant. Then $G$ is nilpotent and its Lie algebra is obtained as a double extension starting from $\{0\}$. In particular, $\nabla J=0$.
\end{corollary}
\begin{proof}
The fact that $G$ becomes nilpotent and its Lie algebra is obtained as a double extension starting from $\{0\}$ are two immediate consequences of Propositions 3.11 and 4.7 from \cite{V} and Theorem \ref{doubleKahler}. Given that, up to isomorphism, the only special K\"ahler Lie group of dimension $2$ is $((\mathbb{R}^{2},+),\omega_0,J_0,\nabla^0)$ and $\mathfrak{g}$ is obtained by a series of double extensions starting from $\{0\}$, we should pass by $\mathbb{R}^2$. Therefore, the fact that $\nabla^0 J_0=0$ and Corollary \ref{NablaJ} imply that $\nabla J=0$.
\end{proof}

\begin{example}
Using $((\mathbb{R}^{2n},+),\omega_0,J_0,\nabla^0)$, the model space of special K\"ahler Lie group, and the double extension process, we can easily get a generic example of non-Abelian special K\"ahler Lie group. Note that the special K\"ahler Lie algebra associated to this Lie group is $(\mathbb{R}^{2n},\omega_0,J_0,\cdot^0)$ where $x\cdot^0 y=0$ since $\nabla^0_{\partial_i}{\partial_j}=0$. If $D\in\mathfrak{sp}(\mathbb{R}^{2n},\omega_0)$ is such that $[D,J_0]=0$, then the vector space $\mathfrak{g}=\mathbb{R}e_{2n+2}\oplus \mathbb{R}^{2n}\oplus \mathbb{R}e_1\cong \mathbb{R}^{2n+2}$ admits a structure of special K\"ahler Lie algebra given by:
\begin{enumerate}
	\item[$\iota.$] the Lie bracket:
	$$[e_1,e_{2n+2}]=[e_{2n+2},e_k]=[e_k,e_{i}]=0\qquad \textnormal{and}\qquad [e_{1},e_k]=D(e_k)$$
	for all $k,i=2,\cdots, 2n+1$. Here $\lbrace e_1,\cdots, e_{2n+2}\rbrace$ denotes the canonical basis of $\mathbb{R}^{2n+2}$ and we are identifying $\mathbb{R}^{2n}$ with $\textnormal{Vect}_\mathbb{R}\lbrace e_2,\cdots, e_{2n+1}\rbrace$. 
	\item[$\iota\iota.$] the non-degenerate scalar 2-cocycle $\widetilde{\omega}$:
	$$\widetilde{\omega}|_{\mathbb{R}^{2n}}=\omega_0,\qquad \widetilde{\omega}(e_{1},e_{2n+2})=-1\qquad\textnormal{and}\qquad \textnormal{Vect}_\mathbb{R}\lbrace e_{1},e_{2n+2}\rbrace \perp_{\widetilde{\omega}} \mathbb{R}^{2n}$$
	\item[$\iota\iota\iota.$] the left symmetric product $\diamond$:
	$$e_k \diamond e_{2n+2}=e_{2n+2}\diamond e_k= e_k\diamond e_i=e_k\diamond e_1=0 \qquad \textnormal{and}\qquad e_1\diamond e_k=D(e_k)$$
	for all $k,i=2,\cdots, 2n+1$, and
	\item[$\iota\nu.$] the integrable complex structure $\widetilde{j}$:
	$$\widetilde{j}|_{\mathbb{R}^{2n}}=J_0\qquad\textnormal{and}\qquad \widetilde{j}(e_{2n+2})=e_1.$$
\end{enumerate}
A Lie group with Lie algebra $\mathfrak{g}$ is $G=\mathbb{R}\ltimes_\rho \mathbb{R}^{2n+1}$ which is determined by the semi-direct product of $(\mathbb{R},+)$ with $(\mathbb{R}^{2n+1},+)$ by means of the Lie group homomorphism $\rho:\mathbb{R}\to \textnormal{GL}(\mathbb{R}^{2n+1})$ defined by
$$\rho(t)=\left(\begin{array}{cc}
\textnormal{Exp}(tD) & 0 \\
0 & 1
\end{array}%
\right).$$
The product in $G$ is explicitly given as 
$$(t,x,u)\cdot(t',x',u)=(t+t',\textnormal{Exp}(tD)(x')+x,u+u').$$
Here $(x,u)$, with $x=(x_2,\cdots,x_{2n+1})\in\mathbb{R}^{2n}$, are the coordinates in $\mathbb{R}^{2n+1}$. A basis for the left invariant vector fields on $G$ are
$$e_1^+=\dfrac{\partial}{\partial t},\qquad e_k^+=\textnormal{Exp}(tD)(e_k)\cdot \left( \dfrac{\partial}{\partial x_2},\cdots, \dfrac{\partial}{\partial x_{2n+1}}\right) \qquad \textnormal{and}\qquad e_{2n+2}^+=\dfrac{\partial}{\partial u},$$
for all $k=2,\cdots, 2n+1$.

It is clear that $J_0\in \mathfrak{sp}(\mathbb{R}^{2n},\omega_0)$. Thus, a particularly interesting choice for the double extension is $D=J_0$. The left invariant vector fields for this particular case are:
$$e_1^+=\dfrac{\partial}{\partial t},\qquad e_k^+=\cos(t)\dfrac{\partial}{\partial x_{k}}+\sin(t)\dfrac{\partial}{\partial x_{n+k}},$$
$$e_{n+k}^+=-\sin(t)\dfrac{\partial}{\partial x_{k}}+\cos(t)\dfrac{\partial}{\partial x_{n+k}} \qquad \textnormal{and}\qquad e_{2n+2}^+=\dfrac{\partial}{\partial u}.$$
for all $k=2,\cdots, n+1$. Therefore, the left invariant special K\"ahler structure on $G=\mathbb{R}\ltimes_\rho \mathbb{R}^{2n+1}$ for the case $D=J_0$ is given by 
\begin{enumerate}
\item[$\iota.$] the left invariant symplectic form  $\displaystyle \omega=du\wedge dt+\sum_{k=2}^{n+1}dx_{k+n}\wedge dx_{k}$,
\item[$\iota\iota.$] the left invariant complex structure $J(e_1^+)=-e_{2n+2}^+$, $J(e_k^+)=e_{n+k}^+$, and $J(e_{n+k}^+)=-e_k^+$, for all $k=2,\cdots, n+1$ and
\item[$\iota\iota\iota.$] the left invariant  flat affine symplectic connection
$$\nabla_{e_1^+}e_k^+=e_{n+k}^+\qquad \nabla_{e_1^+}e_{n+k}^+=-e_k^+,\qquad \nabla_{e_1^+}e_{2n+2}^+=0$$
$$\nabla_{e_{2n+2}^+}=\nabla_{e_k^+}=\nabla_{e_{n+k}^+}=0$$
for all $k=2,\cdots, n+1$.
\end{enumerate}
Given that $\nabla^0$ is geodesically complete and $\textnormal{tr}(J_0)=0$, we obtain that $\nabla$ is geodesically complete as well. Moreover, as we have that $\nabla^0 J_0=0$ and $[J_0,J_0]=0$, we conclude that $\nabla J=0$. The metric on $G$ associated to $(\omega,J)$ is also Riemannian.
\end{example}

We end this section by exhibiting a $1$-dimensional family of left invariant special K\"ahler structures in dimension $6$ with associated metric having signature $(4,2)$ and verifying $\nabla J\neq 0$.
\begin{example}
Consider the special K\"ahler Lie algebra $\mathfrak{g}_3$ in dimension $4$ associated to the special K\"ahler Lie group $G_3$ given in Example \ref{KeyExample1} and whose special K\"ahler structure was described at the beginning of Example \ref{GoodExample}.
Recall that an element $D\in\mathfrak{sp}(\mathfrak{g}_3,\omega)$ such that $[D,j]=0$ has the form:
$$D= \left( 
\begin{array}{cccc}
0 & a & b & c\\
-a & 0 & -c & b\\
b & -c & 0 & d\\
c & b & -d & 0
\end{array}%
\right),\qquad a,b,c,d\in\mathbb{R}.$$
A straightforward computation allows us to deduce that $D(e_1\cdot e_1)=D(e_1)\cdot e_1+e_1\cdot D(e_1)$ if and only if $b=0$, $c=a$, and $d=-a$. Furthermore, if $D_a$ denotes the matrix above after replacing the previous equalities then it is simple to check that this always defines a derivation for the left symmetric product on $\mathfrak{g}_3$. Therefore, we get a $1$-dimensional family of special K\"ahler Lie algebras $\mathfrak{g}_a$ of dimension $6$ paremetrized by $a\in \mathbb{R}$ which are obtained as a double extension from $\mathfrak{g}_3$ according to $D_a$. If $\mathfrak{g}_a\cong \textnormal{Vect}_\mathbb{R}\lbrace e, e_1,e_2,e_3,e_4,d\rbrace$ then this is equipped with

\begin{enumerate}
\item[$\iota.$] nonzero Lie brackets $[e_1,e_2]=[e_1,e_4]=[e_3,e_2]=[e_3,e_4]=e_2-e_4$, and $[d,e_i]=D_a(e_i)$ for all $i=1,\cdots,4$,
\item[$\iota\iota.$] symplectic form $\widetilde{\omega}=e^\ast \wedge d^\ast+e_1^\ast \wedge e_2^\ast - e_3^\ast\wedge e_4^\ast$,
\item[$\iota\iota\iota.$]  integrable complex structure $\widetilde{j}(e)=d$, $\widetilde{j}(e_1)=e_2$, and $\widetilde{j}(e_3)=e_4$; and
\item[$\iota\nu.$] left symmetric product
\begin{center}
\begin{tabular}{c|c|c|c|c|c|c}
$\diamond_a$	& $e$ & $e_1$ & $e_2$ & $e_3$ & $e_4$ & $d$ \\
\hline
$e$	& $0$ & $0$ & $0$  & $0$ & $0$ & $0$ \\
\hline
$e_1$	& $0$ & $-e_1+e_3$ & $e_2-e_4$ & $-e_1+e_3$ & $e_2-e_4$ & $0$ \\
\hline
$e_2$	& $0$ & $0$ & $2e_1-2e_3$  & $0$ & $2e_1-2e_3$ & $0$ \\
\hline
$e_3$	& $0$ & $-e_1+e_3$ & $e_2-e_4$ & $-e_1+e_3$ & $e_2-e_4$ & $0$\\
\hline
$e_4$	& $0$ & $0$ & $2e_1-2e_3$  & $0$ & $2e_1-2e_3$ & $0$ \\
\hline
$d$	& $0$ & $-ae_2+ae_4$ & $ae_1-ae_3$  & $-ae_2+ae_4$ & $ae_1-ae_3$ & $0$ \\
\end{tabular}
\end{center}
\end{enumerate}
Since $\widetilde{L}^a_{e_1}\circ \widetilde{j}\neq \widetilde{j}\circ \widetilde{L}^a_{e_1}$ we have that the left invariant flat affine symplectic connection $\widetilde{\nabla}^a$ determined by $\diamond_a$ satisfies $\widetilde{\nabla}^a\widetilde{J}\neq 0$. Moreover, given that $\mathfrak{g}_3$ is unimodular and $\textnormal{tr}(D_a)=0$, we obtain that $\widetilde{\nabla}^a$ is geodesically complete. Finally, the signature of the scalar product on $\mathfrak{g}_a$ induced by $(\widetilde{\omega},\widetilde{j})$ is $(4,2)$.
\end{example}

\section*{Acknowledgments}
I started this work during a visit to the Mathematisches Institut of Albert--Ludwigs--Universit\"at Freiburg in Freiburg--Germany in February 2020. I am very grateful for the hospitality and support that the Research Training Group 1821 gave me when I was there. I would like to thank Andriy Haydys for pointing out the problem of determining left invariant special K\"ahler structures and Edison Fern\'andez Culma for valuable comments and for having pointed out the infinitesimal version of Example \ref{KeyExample1}. I am also grateful for the partial support given by CODI, Universidad de Antioquia, project 2017-15756 Stable Limit Linear Series on Curves.

I am thankful to the anonymous referee who provided many suggestions and corrections that improved the quality of this paper. Last, but not the least, I would like to express my sincere gratitude to my mother and sister for their support and patience when I was writing the first version of the present work.

\end{document}